\documentclass[10pt]{amsart}
\usepackage{amsfonts}
\usepackage{amsbsy}
\usepackage{tikz}
\usepackage{amsmath}
\usepackage{amssymb}
\usepackage{amsthm}
\usepackage{syntonly}
\usepackage{url} 
\usepackage{amscd} 
\usepackage{tikz-cd} 
\usepackage{bm} 
\usepackage{mathdots} 

\usepackage[colorlinks,linkcolor=brown, anchorcolor=blue, citecolor=green ]{hyperref}
\usepackage{graphicx} 

\usepackage{mathrsfs} 
\usepackage{comment} 
\usepackage{bbm} 

\usepackage[utf8x]{inputenc}
    \usepackage{ucs}
\newcommand\quotient[2]{
        \mathchoice
            {
                \text{\raise1ex\hbox{$#1$}\Big/\lower1ex\hbox{$#2$}}%
            }
            {
                #1\,/\,#2
            }
            {
                #1\,/\,#2
            }
            {
                #1\,/\,#2
            }
    }

\newcommand{\R}{\mathbb{R}} 
\newcommand{\Z}{\mathbb{Z}}
\newcommand{\Q}{\mathbb{Q}}
\newcommand{\C}{\mathbb{C}}

\newcommand{\bbH}{\mathbb{H}}
\newcommand{\bbI}{\mathbb{I}}

\newcommand{\bma}{\bm{a}}
\newcommand{\bmA}{\bm{A}}
\newcommand{\bmB}{\bm{B}}
\newcommand{\bmc}{\bm{c}}
\newcommand{\bmC}{\bm{C}}
\newcommand{\bmD}{\bm{D}}
\newcommand{\bme}{\bm{e}}
\newcommand{\bmf}{\bm{f}}
\newcommand{\bmF}{\bm{F}}
\newcommand{\bmG}{\bm{G}}
\newcommand{\bmH}{\bm{H}}
\newcommand{\bmL}{\bm{L}}
\newcommand{\bmM}{\bm{M}}
\newcommand{\bmN}{\bm{N}}
\newcommand{\bmo}{\bm{o}}
\newcommand{\bmp}{\bm{p}}
\newcommand{\bmP}{\bm{P}}
\newcommand{\bmS}{\bm{S}}
\newcommand{\bmT}{\bm{T}}
\newcommand{\bmU}{\bm{U}}
\newcommand{\bmv}{\bm{v}}
\newcommand{\bmV}{\bm{V}}
\newcommand{\bmW}{\bm{W}}
\newcommand{\bmX}{\bm{X}}
\newcommand{\bmY}{\bm{Y}}
\newcommand{\bmZ}{\bm{Z}}

\newcommand{\frakg}{\frak{g}}
\newcommand{\frakh}{\frak{h}}

\newcommand{\frakz}{\frak{z}}

\newcommand{\scrA}{\mathscr{A}}

\newcommand{\scrD}{\mathscr{D}}
\newcommand{\scrF}{\mathscr{F}}
\newcommand{\scrP}{\mathscr{P}}

\newcommand{\DD}{\mathcal{D}}

\newcommand{\calD}{\mathcal{D}}
\newcommand{\calF}{\mathcal{F}}
\newcommand{\calH}{\mathcal{H}}
\newcommand{\calI}{\mathcal{I}}

\newcommand{\calO}{\mathcal{O}}
\newcommand{\calP}{\mathcal{P}}
\newcommand{\calR}{\mathcal{R}}
\newcommand{\calW}{\mathcal{W}}

\newcommand{\widev}{\widetilde{v}}
\newcommand{\whmu}{\widehat{\mu}}

\newcommand{\ep}{\varepsilon}

\newcommand{\bs}{\backslash}  
\newcommand{\la}{\langle}
\newcommand{\ra}{\rangle}

\newtheorem{thm}{Theorem}[section]
\newtheorem{conj}[thm]{Conjecture}
\newtheorem{coro}[thm]{Corollary}
\newtheorem{defi}[thm]{Definition}

\newtheorem{lem}[thm]{Lemma}
\newtheorem{prop}[thm]{Proposition}

\newtheorem{rmk}[thm]{Remark}

\DeclareMathOperator{\SL}{SL}
\DeclareMathOperator{\GL}{GL}

\DeclareMathOperator{\SO}{SO}
\DeclareMathOperator{\Sp}{Sp}
\newcommand{\fraksl}{\frak{sl}}

\newcommand{\fraksp}{\frak{sp}}

\DeclareMathOperator{\Ad}{Ad}

\DeclareMathOperator{\Cone}{Cone}
\DeclareMathOperator{\diff}{d}
\DeclareMathOperator{\diffa}{da}
\DeclareMathOperator{\diffA}{dA}
\DeclareMathOperator{\diffb}{db}
\DeclareMathOperator{\diffB}{dB}

\DeclareMathOperator{\diffh}{dh}
\DeclareMathOperator{\diffk}{dk}

\DeclareMathOperator{\difft}{dt}

\DeclareMathOperator{\diffu}{du}
\DeclareMathOperator{\diffx}{dx}
\DeclareMathOperator{\diffy}{dy}
\DeclareMathOperator{\diffz}{dz}
\DeclareMathOperator{\diffZ}{dZ}
\DeclareMathOperator{\diffalpha}{d\alpha}
\DeclareMathOperator{\diffmu}{d\mu}
\DeclareMathOperator{\diffnu}{d\nu}
\DeclareMathOperator{\diag}{diag}

\DeclareMathOperator{\Lie}{Lie}

\DeclareMathOperator{\Map}{Map}

\DeclareMathOperator{\Top}{Top}
\DeclareMathOperator{\Tr}{Tr}
\DeclareMathOperator{\Vol}{Vol}

\begin{document}

\title[Translates of homogeneous measures]{Translates of homogeneous measures associated with observable subgroups on some homogeneous spaces}
\author[R.Zhang]{Runlin Zhang}


\begin{abstract}
    In the present article we study the following problem. Let $\bmG$ be a linear algebraic group over $\Q$, $\Gamma$ be an arithmetic lattice and $\bmH$ be an observable $\Q$-subgroup. There is a $H$-invariant measure $\mu_H$ supported on the closed submanifold $H\Gamma/\Gamma$. Given a sequence $\{g_n\}$ in $G$ we study the limiting behavior of $(g_n)_*\mu_H$ under the weak-$*$ topology. In the non-divergent case we give a rather complete classification. We further supplement this by giving a criterion of non-divergence and prove non-divergence for arbitrary sequence $\{g_n\}$ for certain large $\bmH$. We also discuss some examples and applications of our result. This work can be viewed as a natural extension of the work of Eskin--Mozes--Shah and Shapira--Zheng.
\end{abstract}

\maketitle

\tableofcontents

\section*{Convention}

We will use the following notations throughout the paper.

\begin{defi}
A \textbf{standard triple} $(\bmG,\bmH,\Gamma)$ consists of the following data:
\begin{itemize}
    \item $\bmG$ is a connected linear algebraic group defined over $\Q$;
    \item $\bmH$ is a connected $\Q$-subgroup of $\bmG$;
    \item $\Gamma\leq \bmG(\Q)$ is an arithmetic subgroup commensurable with $\bmG(\Z)$;
\end{itemize}
If furthermore $\bmH$ is an observable subgroup of $\bmG$, then we say that the standard triple is observable.
\end{defi}

 To simplify notations we use $\bmG$ also for $\bmG(\C)$  for a linear algebraic group $\bmG$ over $\C$. If $\bmG$ is defined over $\R$, then the corresponding Roman letter $G$ denotes the analytic identity connected component of $\bmG(\R)$. And $\Gamma_H$ is defined to be $\Gamma \cap H$ for a $\Q$-subgroup $\bmH\leq \bmG$. We always assume $\Gamma$ to be contained in $G$ and write $\pi_{\Gamma}$ for the natural projection $G \to G/\Gamma$.

Given a standard triple $(\bmG,\bmH,\Gamma)$,
there exists a left $H$-invariant measure supported on $H/H\cap \Gamma$(see \cite[Lemma 1.4]{Ragh72}), which is denoted by $\mu_{H}$. If $\bmH$ is observable in $\bmG$, then the natural map from $H/H\cap \Gamma$ to $H\Gamma/\Gamma$ is a closed embedding into $G/\Gamma$ and we may push $\mu_H$ to a locally finite measure on $G/\Gamma$ supported on $H\Gamma/\Gamma$.
For a non-empty open bounded subset $\calO\subset H$, let $\mu_{\calO}$ denote the restriction of $\mu_H$ to $\pi_{\Gamma_H}(\calO)$. One does not need $\bmH$ to be observable to push $\mu_{\calO}$ to a (finite) measure on $G/\Gamma$.

Consider the collection of non-zero locally finite positive measures on a locally compactly second countable space $X$. Two such measures $\mu$ and $\nu$ are said to be equivalent iff there exists a positive real number $a>0$ such that $a\mu=\nu$. 
The equivalence class containing $\mu$ is denoted by $[\mu]$. Whenever a measure $\mu$ is known to be finite we let $\widehat{\mu}$ denote the unique probability measure in the equivalence class containing $\mu$.
A sequence of classes $\{[\mu_n]\}_{n\in \Z^+}$ is said to
converge to 
$[\nu]$ 
if and only if one of the following equivalent conditions is satisfied (see \cite[Proposition 3.3]{ShaZhe18}):
\begin{enumerate}
    \item  For all $f_1,f_2\in C_c(X)$, compactly supported continuous functions on $X$ such that $(f_2,\nu)\neq 0$ we have
\begin{equation*}
    \lim_{n\to \infty}
    \frac{ \int f_1(x) \diffmu_n(x)  }  
    {\int f_2(x) \diffmu_n(x)} =\frac{\int f_1(x) \diffnu(x)}{\int f_2(x) \diffnu(x)}.
\end{equation*}
    \item There exists a sequence of positive real numbers $\{a_n\}$ such that for all $f\in C_c(X)$, $\frac{1}{a_n} \la f,\mu_n\ra\to \la f, \nu \ra$.
\end{enumerate}
The sequence $\{a_n\}$ is also interesting so sometimes we keep track of this too. Note that by \cite[Proposition 3.3]{ShaZhe18}, the asymptotic of $\{a_n\}$ is uniquely determined once representatives $\mu_n$, $\nu$ are fixed.

\section{Introduction}

In the present article we study the following problem.
Given a standard triple $(\bmG,\bmH,\Gamma)$ and a sequence $\{g_n\}$ in $G$, what is the limit of $(g_n)_*\mu_H$ under the weak-$*$ topology? The original interest in such a problem comes from the study of the asymptotics of integer points on an affine homogeneous variety. After the pioneering work of Duke--Rudnick--Sarnak \cite{DukRudSar93} where harmonic analysis method is used, Eskin--McMullen \cite{EskMcM93}(compare \cite{BenOh12}) give a simpler proof using mixing. They assume that $\bmH$ is symmetric, i.e., consisting of the fixed point of some involution, and has no non-trivial $\Q$-characters. The latter condition is equivalent to that $\Gamma_H$ is of finite covolume in $H$. Based on the unipotent rigidity theorem of Ratner \cite{Rat91} and linearization technique developed by Dani--Margulis \cite{DanMar93}, Eskin--Mozes--Shah in \cite{EskMozSha96} and \cite{EskMozSha97} make a non-effective generalization assuming $\Gamma_H$ to be of finite covolume in $H$. On the one hand, they prove that, assuming the non-divergence of $(g_n)_*\mu_H$, any limit measure has to be a homogeneous measure. On the other hand, they complement this by showing that when $\bmG$ and $\bmH$ are both reductive and $\bmH$ is not contained in any proper $\Q$-parabolic subgroup of $\bmG$, then non-divergence of $(g_n)_*\mu_H$ automatically holds for each sequence $\{g_n\}$ in $G$.
In recent years, there are interests in removing the condition of $\Gamma_H$ being of finite covolume in $H$. In the work of Oh--Shah \cite{OhSha14}, such a generalization is obtained for $\bmG=\SL_2$ and $\bmH$ equal to the diagonal torus. A different proof is given later by Kelmer--Kontorovich \cite{KelKon18}(c.f. \cite{KelKon182}) which yields stronger result. These two results are effective. Shapira--Zheng \cite{ShaZhe18} generalize the original approach of \cite{EskMozSha96} to treat the case when $\bmG=\SL_n$ and $\bmH$ is a maximal $\Q$-split torus. From their work, the key is to define certain family of polytopes and show that they grow in all directions. Zhang \cite{Zha18} further generalizes their work by allowing $\bmH$ to be an arbitrary maximal $\Q$-torus in $\SL_n$(or after applying restriction of scalar to them). The main new difficulty there is to show that the polytopes defined indeed give non-divergence.

On the other hand one should be careful when dropping the condition that $\bmH$ has no non-trivial $\Q$-characters. For instance, when $\bmG$ is $\SL_2$ with the standard $\Q$-structure and $\bmH$ is the subgroup of upper triangular matrices, by duality one sees that $\pi_{\Gamma}(H)$ is dense in $G/\Gamma$. Hence $\mu_H$, even if defined, would not be a locally finite measure. So it does not live inside the dual of compactly supported continuous functions.
A sufficient group theoretical condition to guarantee the closedness of $\pi_{\Gamma}(H)$ is that $\bmH$ is an observable subgroup of $\bmG$. The converse is also true and is due to Barak Weiss \cite[Corollary 7]{Weiss98}. It is healthy to keep this in mind, but we shall not make use of this fact and would actually deduce it from our analysis(see Corollary \ref{coroWeiss}).

One may also decide not to consider the full orbit of $H$, but rather a bounded piece. And one does not have to require $\bmH$ to be defined over $\Q$. See the work of Richard--Zamojski \cite{RicZam17} in the case when $\bmH$ is reductive.

One may also consider the similar question in the adelic setting. As we shall not touch upon this, the reader is referred to \cite{Zam10}, \cite{EinMarMohVen15}, \cite{GorOhMau08}, \cite{GoroOh11}, \cite{DavSha18} and \cite{DavSha19} for more information.

Let us start with definitions of an observable subgroup. All representations are assumed to be finite-dimensional (algebraic) linear representations. It is helpful to keep in mind a theorem of Chevalley which asserts that one can always find a representation of $\bmG$ and a \textit{line} whose stabilizer is equal to $\bmH$.

\begin{defi}\label{defiObs}
Let $\bmG$ be a linear algebraic group over $\Q$ and $\bmH$ be a $\Q$-subgroup. $\bmH$ is said to be an 
\textbf{observable subgroup} of $\bmG$ if and only if one of the following equivalent conditions is satisfied:
\begin{enumerate}
    \item there exists a $\Q$-representation $(\rho,\bmV)$ of $\bmG$ and a non-zero vector $v\in \bmV(\Q)$ such that $\bmH$ is the stabilizer of $v$ in $\bmG$;
    \item same statement as in (1) replacing $\Q$ by $\overline{\Q}$;
    \item each $\Q$-representation  $(\rho_0,\bmV_0)$ of $\bmH$ is contained in a $\Q$-representation of $\bmG$;
    \item same statement as in $(3)$ replacing $\Q$ by $\overline{\Q}$;
     \item for any one-dimensional $\Q$-representation  $(\rho_0,\bmV_0)$ of $\bmH$ that is contained in some $\Q$-representation of $\bmG$, the dual of $\rho_0$ is also contained in some $\Q$-representation of $\bmG$;
     \item same statement as in $(5)$ replacing $\Q$ by $\overline{\Q}$.
\end{enumerate}
We say that a $\Q$-representation $\rho: \bmG \to \SL(\bmV)$ is observable if and only if the image of $\bmG$ is an observable subgroup of $\SL(\bmV)$.
\end{defi}

The equivalences between $(3)$ and $(5)$ and between $(4)$ and $(6)$ are proved in \cite[Theorem 1]{BiaHochMos63}. 
The equivalence between $(3)$ and $(4)$ is proved in \cite[Theorem 5]{BiaHochMos63}. 
The equivalence between $(2)$ and $(4)$ is proved in \cite[Theorem 8]{BiaHochMos63} where it is also proved that $(3)$ implies $(1)$. As $(1)$ trivially implies $(2)$ we have all the equivalences above. 

Examples of observable $\Q$-subgroups include all reductive groups(see \cite[Corollary 2.4]{Grosshans97}) and all groups with no non-trivial $\Q$-characters. Proper parabolic subgroups $\bmH$ are always examples of non-observable subgroups as otherwise there would be non-constant global (regular) functions on $\bmG/\bmH$.
Also, being observable is a relative notion. For instance, every group is an observable subgroup of itself. We shall also show that any $\Q$-linear group can be realized as an observable subgroup of some $\SL_n$(see Lemma \ref{lemObsRepExi}). 

If $(\bmG,\bmH,\Gamma)$ is an observable standard triple and $v\in \bmV(\Q)$ is as in the above definition, then $\Gamma \cdot v $ is discrete and hence closed in $G\cdot v$. 
As the map $g \mapsto g\cdot v$ induces a continuous map from $G/H$ and $G\cdot v$, by taking preimage we see that $\Gamma H$(note this may not be exactly the preimage but the argument still works) is closed in $G$ hence $\pi_{\Gamma}(H)$ is closed in $G/\Gamma$. Therefore, the natural map from $H/\Gamma_H$  to $H\Gamma/\Gamma$ is a homeomorphism and so the latter supports a unique up-to-a-scalar locally finite $H$-invariant measure(see \cite[Lemma 1.7]{Ragh72}), which we also denote by $\mu_H$.

Now we take a sequence $\{g_n\}$ in $G$ and ask what the possible limits of $(g_n)_*[\mu_H]$ are.

\subsection{Main theorem and its refinements}
\begin{thm}\label{transbyg_n}
Let $(\bmG,\bmH,\Gamma)$ be an observable triple and $\{g_n\}$ be a sequence in $G$. Then either one of the following is true:\\
(1) $\pi_{\Gamma}(g_nH)$ diverges set-theoretically, that is to say, for any compact set $C\subset G/\Gamma$, there exists $n_C$ such that for all $n\geq n_C$, $\pi_{\Gamma}(g_nH)\cap C=\emptyset$;\\
(2) there exists a subsequence $\{g_{n_k}\}$, an element $\delta\in G$ and an observable $\Q$-subgroup $\bmL$ of $\bmG$ such that $\lim_k (g_{n_k})_*[\mu_H]= (\delta)_*[\mu_{L}]$. In other words, there exists a sequence of $\{a_k\}\subset \R^+$ such that $\lim_k \frac{1}{a_k}(g_{n_k})_*\mu_H= (\delta)_*\mu_L$.
\end{thm}

Suppose that we are in the second case and have already passed to the subsequence. 
Then we wish to know how to pin down what $\bmL$ is. To see this, take a compact set $C$ that intersects non-trivially with $\pi_{\Gamma}(g_nH)$ for all $n$. It follows that one can write $g_n = \delta_n \gamma_n h_n$ for some sequence $\{\delta_n\}$ bounded in $G$, $\{\gamma_n\} \subset \Gamma $ and $\{h_n\}\subset H$. Then $(g_{n})_*\mu_H=(\delta_n\gamma_n)_*\mu_H$. By passing to a further subsequence one may assume that $\lim \delta_n = \delta$(this $\delta$ turns out to be the same $\delta$ as appeared in the last theorem). So it remains to understand the limit of $(\gamma_n)_*\mu_H$.

\begin{thm}\label{thmTranslatebyGamma}
Let $(\bmG,\bmH,\Gamma)$ be an observable triple and $\{\gamma_n\}$ be a sequence in $\Gamma$. Assume that $\bmL$ is an observable $\Q$-subgroup of $\bmG$ with $(\{\gamma_n\},\bmL)$ being potentially minimal for $\bmH$. Then $\lim_n (\gamma_{n})_*[\mu_H]=[\mu_{L}]$.
\end{thm}

Potentially minimal means that $\bmL$ contains conjugates of $\bmH$ by $\gamma_n$ for all $n$ and is minimal among observable subgroups that contains $\gamma_n \bmH \gamma_n^{-1}$ for infinitely many $n$(see Definition \ref{defminimal}). Such an $\bmL$ always exists after passing to a subsequence. Note that this theorem says that the convergence is essentially from ``inside'' and this fact implies the following topological statement. The reader is referred to \cite[Chapter E.1]{BenePetr92}, especially Proposition E.1.2, for the definition and basic properties of Chabauty topology.

\begin{coro}\label{thmConChabau}
Notation as in Theorem \ref{transbyg_n} above and assume that case (1) does not happen. Then in the Chabauty topology $\pi_{\Gamma}(g_{n_k}H)$ converges to $\pi_{\Gamma}(\delta L)$.
\end{coro}

We also wish to know how to compute the asymptotic of $\{a_n\}$ as this is uniquely determined once a normalization of $\mu_H$ and $\mu_L$ is fixed.
When $\mu_H$ is finite, $\{a_n\}$ can be taken to be a constant sequence. In general, let $\bmS_{\bmH}$ be its maximal quotient $\Q$-split torus($\bmS_{\bmH}=\bmH$ if $\bmH$ itself is a $\Q$-split torus) of $\bmH$. Then the projection $H\to S_{\bmH}$ is a principal ${}^{o}H$(see Section \ref{secPolytopes} for definitions)-homogeneous bundle over $S_{H}$ which factors through $\pi_{{}^{\circ}\bmH}: H/H\cap \Gamma \to S_{\bmH}$ with fibres isomorphic to (translates of) ${}^{o}H/{}^{o}H\cap \Gamma$ which have finite volume. Note that for a bounded set $\calR$ in $S_{H}$, its preimage $\calP(\calR)$ in $H$ is right-invariant by $\Gamma_H$ and has finite measure in the quotient. And if $\calR$ is compact then $\calP(\calR)/\Gamma_H$ regarded as contained in $H\Gamma/\Gamma$ is closed in $G/\Gamma$ even when $\bmH$ is not observable. Then we hope to find a sequence of $\calR_n$ such that the contribution of $(g_{n})_*[\mu_H]$ to an integration of a compactly supported function ``essentially'' comes from its restriction to $\calP(\calR_n)/\Gamma_H$. Then we wish to set $a_n$ to be the $\mu_H$-measure of $\calP(\calR_n)/\Gamma_H$.  

Of course, a necessary condition for $\calP(\calR_n)/\Gamma_H$ to contribute to the integral is that it comes back to a compact set containing the support of the function. To give a nice family of compact sets we make use of Mahler's criterion. So take a faithful observable $\Q$-representation $\rho$ of $\bmG$ into some $\SL_N$, whose existence will be guaranteed in Section \ref{secRepandNondiv}. We may assume that $\Gamma$ preserves the standard $\Z^N$ in the representation. Then we have $G/\Gamma\to \SL_N(\R)/\SL_N(\Z)$ a proper map.
According to Mahler's criterion, for $\eta>0$,
\begin{equation*}
    K_{\eta}(\rho):=\{[g]\in G/\Gamma, ||\rho(g)v|| \geq \eta ,\,\forall v_{\neq 0} \in\Z^N\}
\end{equation*}
would be a nice family of compact sets. So naively, one might define (for some $\eta>0$) the region $\calP_n$, which may not be of the form $\calP(\calR_n)$, to be those $h$ such that  $||g_nhv||>\eta$ for all non-zero integral vectors $v$. This would be a good definition except that:
\begin{itemize}
    \item to prove our main theorems above, we will need to go ``deeper'' in this $\calP_n$. And to justify our arguments we will need $\calP_n$ to be of the form $\calP(\calR_n)$  with $\calR_n$ being (the image under exponential map of) \textit{convex polytopes};
    \item when applied to counting problems, we would like to explicitly compute the asymptotics of $a_n$ for ``generic'' $\{g_n\}$ and with this naive definition it is difficult.
\end{itemize}

To get such a region we relax ourselves to demand $||g_nhv||>\eta$ only for all integral \textit{weight vectors} with respect to $\bmH$. Now $\calP_n=\calP(\calR_n)$ are nice except that we may lose the non-divergence. We will remedy this by
\begin{itemize}
    \item looking at more weight vectors from a bigger representation $\rho''$;
    \item allowing perturbation by a fixed nonempty open bounded set $\calO$ in $H$.
\end{itemize}
Then we will define  $\calP_n=\calP(g,\eta,\rho'',\Phi_{\rho''})$(see Definition \ref{defPolytope} and Equation \ref{defPreImaPoly} for the definition) and prove the following.
The reader is referred to Section \ref{secRepandNondiv} for precise definitions of undefined terms below.

\begin{thm}\label{thmNondiv}
Given a standard triple $(\bmG,\bmH,\Gamma)$. Let $\rho''$ be a superfaithful $\Q$-representation of $\bmG$. Take $\calO$ to be a non-empty open bounded subset of $H$. For each sequence $\{g_n\}$ in $G$, $\eta>0$ and $h_n$ in $\calP(g_n,\eta,\rho'',\Phi_{\rho''})$, all weak-$*$ limits of $\{(g_nh_n)_*\mu_{\calO}\}$ in $G/\Gamma$ have the same total mass as $\mu_{\calO}$.
\end{thm}

There is a natural question about whether it is possible to reduce the amount of weight vectors that are needed. This will make the computation of $a_n$ easier in practice. We will state what we think is true and provide some evidence towards this in Section \ref{secRepandNondiv}.  

Now we can state what $a_n$ is in the special case when the limit measure is known to be finite. The reader is referred to Section \ref{secGammaEqui} for general cases.

\begin{thm}\label{thmCompA_n}
In the set-up of Theorem \ref{thmTranslatebyGamma} above. Assume also that $\mu_L$ is finite and normalized to be a probability measure. Then we may take 
$$a_n:= \mu_H(\calP(\gamma_n,\eta,\rho'',\Phi_{\rho''})/\Gamma_H)$$
 where $\eta>0$ is any small enough positive number. That is to say, 
 $\lim \frac{1}{a_n}(\gamma_n)_*\mu_H = \mu_L$.
\end{thm}
Similarly there is a corresponding statement for Theorem \ref{transbyg_n}.

In the set-up of Theorem \ref{transbyg_n}, there are certain cases when non-divergence is automatically guaranteed for all possible sequences $\{g_n\}$. Indeed, when $\bmG$ is reductive and $\bmZ_{\bmG}\bmH$ is assumed to be $\Q$-anisotropic, this has been achieved in \cite[Theorem 1.9]{EskMozSha96}. Note that under these conditions, $\mu_H$ is implied to be finite. We prove a generalization when $\mu_H$ may not necessarily be finite. 
Maximal $\Q$-split tori in a reductive $\Q$-group $\bmG$ would be examples of such an $\bmH$ and there are also examples that do not contain a maximal $\Q$-split torus. 

\begin{thm}\label{thmTransReduc}
Same notations as in Theorem \ref{transbyg_n} or \ref{thmTranslatebyGamma} above. We assume in addition that $\bmG$ and $\bmH$ are both reductive and $\bmZ_{\bmG}\bmH/\bmZ_{\bmG}\bmH\cap \bmH$ is $\Q$-anisotropic. Then 
\begin{itemize}
    \item case (1) in Theorem \ref{transbyg_n} never happens;
    \item $\bmL$ is reductive.
\end{itemize}
Moreover, if $\{\gamma_n\}$(or equivalently, $\{g_n\}$) diverges in $G/Z_{\bmG}\bmS$ for all $\Q$-split tori $\bmS$ contained in the center of $\bmH$, then $\bmL$ is not contained in any proper $\Q$-parabolic subgroup of $\bmG$ and $\bmG$ has no $\Q$-character. In particular, $\mu_L$ is finite.
\end{thm}
Note that conversely, if $\bmZ_{\bmG}\bmH/\bmZ_{\bmG}\bmH\cap \bmH$ is not $\Q$-anisotropic, then case (1) in Theorem \ref{transbyg_n} does happen for certain $\{g_n\}$.

In the process of showing Theorem \ref{thmTransReduc}, we also obtain the following group-theoretical result(note that this is not a corollary of the unipotent rigidity theory):

\begin{thm}\label{corObConRedisRed}
For two reductive $\Q$-groups $\bmH\leq \bmG$. Assume the center of $\bmG$ is $\Q$-anisotropic.
Then all observable $\Q$-subgroups of $\bmG$ that contain $\bmH$ are reductive if and only if $\bmZ_{\bmG}\bmH/\bmZ_{\bmG}\bmH\cap \bmH$ is $\Q$-anisotropic.
\end{thm}

Note that ``if'' is the nontrivial direction.
In the special case when $\bmG$ is $\Q$-anisotropic, this is easy as there is no rational unipotent element. 
In the special case when $\bmH$ is a maximal torus, this is proved in \cite[Lemma 3.10]{Grosshans97} where they also allow $\bmG$ to be non-reductive, which is not covered here.

\subsection{Examples and Applications}

Now we turn to more explicit situations and some applications. 

\begin{prop}\label{propMaxTori}
Keep the notations as in Theorem \ref{thmTransReduc}. Suppose that $\bmH$ is a maximal torus in $\bmG$ and for all diagonalizable $\Q$-subgroups $\bmS$ of $\bmH$ that properly contain $\bmZ(\bmG)$, one has $\{g_n\}$(or equivalently, $\{\gamma_n\}$) diverges in $G/Z_{G}(S)$. Then $\bmL=\bmG$.
\end{prop}

Note that in general it would be insufficient only to consider those subtori(namely, connected) $\bmS$. A counter example may be found in $\SO_{2N+1}$, where $\SO_{2N}$ is the centralizer of a disconnected diagonalizable subgroup. However this is sufficient if $\bmG=\SL_N$ as was shown in \cite{ShaZhe18} and \cite{Zha18}. We will show that this also holds for the symplectic group $\Sp_{2N}$. For simplicity we only treat the ``generic'' case.

\begin{prop}\label{propToriSym}
Keep the notations as in Theorem \ref{thmTransReduc}. Suppose $\bmG=\Sp_{2N}$ and $\bmH$ is a maximal $\Q$-torus. If $\{g_n\}$ diverges in $G/Z_{G}(S)$ for all nontrivial subtori $\bmS$ of $\bmH$, then $\bmL=\bmG$.
\end{prop}

Using this we shall prove a counting result. 
Fix a polynomial $p(t)$ of the form $\prod_{i=1}^N (t^2 -d_i^2)$ with $d_i\in \Z^+$ distinct. Consider
$\bmX:= \{X\in \fraksp_{2N},\, \det(tX-I_{2N})=p(t) \}$. 
Let $||\cdot||$ be the Euclidean norm on $2N$-by-$2N$ matrices. 

\begin{thm}\label{thmCountSym}
There exists a constant $C>0$ such that 
\begin{equation*}
    \lim_{R\to\infty}  \frac{\#\{X\in \bmX(\Z),\, ||X||\leq R \}}{CR^{N^2}(\ln{R})^{N}} =1.
\end{equation*}
\end{thm}
A better, albeit not very explicit, description of $C$ will be given before Proposition \ref{PropSymCount}.

Next we turn to a geometric application.
Let $\bmH^n\subset \R^n$ be the upper half space model of the $n$-dimensional hyperbolic space. 
Let $Q$ be the quadratic form $x_1^2+...+x_n^2-y^2$ and
$\bmG=\SO_{Q}$ be its symmetric group. We assume $\Gamma$, commensurable with $\bmG(\Z)$, to be contained in $G$. Then $\Gamma$ naturally acts as isometries on $\bmH^n$, so we may form the quotient $\Gamma\bs\bmH^n$ as a metric space. By abuse of notation, let $\pi_{\Gamma}$ also denote the projection  $\bmH^n\to\Gamma\bs\bmH^n$.
The set $\pi_{\Gamma}(\{(0,...,0,t),\,t\in \R\})$ is (the base locus of) a divergent geodesic in the quotient  $\Gamma\bs\bmH^n$.  Take a non-zero vector $\bmv=(v_1,...,v_{n-1})\in \R^{n-1}$, which we identify with the boundary of $\bmH^n$ in $\R^n$. Now we start to ``shear'' the geodesic by looking at the projection of $\calI_{s\bmv}:=\{t(sv_1,...,sv_{n-1},1),\,t\in \R\}$ as $s$ tends to $\infty$. For simplicity we only state the generic case. 

\begin{thm}\label{thmShearGeod}
Assume that $\bmv$ is not contained in any proper $\Q$-linear subspace. 
Then under the Chabauty topology, $\pi_{\Gamma}(\calI_{s\bmv})$ converges to $\Gamma\bs\bmH^n$ as $s$ tends to $\infty$.
\end{thm}

In the case when there is only one cusp, it should be possible to upgrade the convergence to be under the Hausdorff distance, which is in general stronger than Chabauty topology.



\subsection{Organization of the paper}
In section \ref{secRepandNondiv} we prove Theorem \ref{thmNondiv}. This is based on a trick of taking exterior powers. In section \ref{secNondivMaxSplTori} we prove a stronger result in the case of maximal split tori.

In section \ref{secGammaEqui} we prove Theorem \ref{transbyg_n}, \ref{thmTranslatebyGamma} and \ref{thmCompA_n}. They are corollaries of Proposition \ref{partIIofsection3}.
Some basic facts on polytopes and cones are collected in Section \ref{secPolytopes}.
Then we recall \cite[Theorem 2.1]{EskMozSha96} and enhance the statement in Section \ref{sec3PartI}.
In Section \ref{sec3PartII} we utilize the notion of observability to conclude the proof. 

In section \ref{secRed} we prove Theorem \ref{thmTransReduc} and \ref{corObConRedisRed}. Besides Theorem  \ref{thmTranslatebyGamma}, we also need the work of Kempf \cite{Kemp78}. 
We will also give an alternative short proof of non-divergence in this situation based on the real version of a theorem of Ness \cite{Nes84}(see \cite{Wal17}).

In the last section \ref{SecExamAppl} we prove the examples and applications listed above.

\section{Representation and non-divergence}

\subsection{The general case}\label{secRepandNondiv}
Let $\bmG$ be a connected linear algebraic group and $\rho: \bmG \to \GL(\bmV)$ be a representation over $\Q$. We let $X^*(\bmG)$(resp. $X_*(\bmG)$) be the $\Z$-module of $\Q$-characters(resp. $\Q$-cocharacters) of $\bmG$.
For each $\alpha \in X^*(\bmG)$, let
\begin{equation*}
    \bmV_{\alpha}:=\{v\in \bmV \,\vert \, gv=\alpha(g)v,\,\forall g\in \bmG \},
    \quad
    \Phi_{\rho}:=\{\alpha\in X^{*}(\bmG)\, \vert \, \bmV_{\alpha}\neq \{0\} \}.
\end{equation*}
For $\alpha\in X^*(\bmG)$, $\bmV_{\alpha}$ is defined over $\Q$.

Let (see \cite[Section 1.1]{BorSer73})
\begin{equation*}
    {}^{\circ}\bmG:= \bigcap_{\alpha \in X^*(\bmG)}  \ker (\alpha^2),
    \quad
    \bmS_{\bmG}:= \bmG / {}^{\circ}\bmG,
\end{equation*}
then $\bmS_{\bmG}$ is a $\Q$-split torus.
Let $\pi_{{}^{\circ}\bmG}:\bmG \to \bmS_{\bmG} $ be the natural projection.
Note that for $s \in S_{\bmG}$ and $\alpha \in X^*(\bmG)$, $\alpha(s)$ is well defined. 
This is because $\alpha$ is trivial on ${}^{\circ}G$(indeed, $\alpha(G)$ are positive numbers and so $\alpha(g)^2=1$ implies that $\alpha(g)=1$ for $g\in G$) and we also claim that

\subsubsection*{Claim} $S_{\bmG} = G/{}^{\circ}G $. 

\begin{proof}
There is a natural map $S_{\bmG}\to G/ G\cap {}^{\circ}\bmG(\R) \to G/{}^{\circ}\bmG$. As the latter two are connected commutative Lie groups, they are isomorphic to $(\R^l_{>0}, *)$ for some $l$. As this map is surjective on Lie algebras, it is also surjective on groups as $\R^l_{>0}$ contains no proper open subgroup. It is by definition that the first arrow on the left is injective. The next one would have kernel being finite, as ${}^{\circ}\bmG(\R)$ only has finitely many components in analytic topology. But $\R^l_{>0}$ contains no finite subgroup other than the trivial group. So we are done.
\end{proof}

Therefore $\diffalpha(t):=\ln \alpha(\exp{(t)})$ is also well-defined for $\alpha\in X^*(\bmG)$ and $t \in \Lie(S_{\bmG})$.

Now take $\bmH$ to be a $\Q$-subgroup of $\bmG$.
For a subset $\Phi\subset \Phi_{\rho}$, a $\Z$-structure on $\bmV_{\Q}$(the dependence on which we often suppress), an element $g\in G$ and a positive real number $\ep$ we define a polytope in the Lie algebra of $S_{\bmH}$ by
\begin{defi}\label{defPolytope}
\begin{equation*}
\begin{aligned}
        \Omega(g,\ep,\rho,\Phi):=&
    \{
    t\in \Lie(S_{\bmH}) \mid
    \inf_{0\neq v \in \bmV_{\alpha}(\Z)} ||g\exp{(t)}v||\geq \ep, \,\forall \alpha \in \Phi
    \}\\
    =&
    \{
    t \in \Lie(S_{\bmH})  \mid
    \diffalpha(t)\geq \ln\ep -\ln \inf_{0\neq v \in \bmV_{\alpha}(\Z)}||gv||,\, \forall \alpha\in \Phi
    \}.
\end{aligned}
\end{equation*} 
\end{defi}

 Note that $\exp{(t)}v$ for $v$ in $\bmV_{\alpha}$ is well defined up to $\pm$ sign. These definitions make sense even when $S_{\bmG}=\{e\}$ in which case $\Omega$ is either $\{0\}$ or empty.

There is a slightly different situation that we shall encounter later in Section \ref{secGammaEqui}. 
Take $(\bmG,\bmH,\Gamma)$ to be a standard triple. Let $\bmL$ be another connected $\Q$-subgroup of $\bmG$. Let $\bmX(\bmH,\bmL)$ be the set of $g \in \bmG$ such that $g \bmH g^{-1}$ is contained in $\bmL$. 
For each $\gamma\in \bmX(\bmH,\bmL)\cap \Gamma$, let $\bmc_{\gamma}$ be the morphism from $\bmH$ to $\bmL$ defined by $h\mapsto \gamma h \gamma^{-1}$. Now take a $\Q$-representation $\rho$ of $\bmL$.
To defined the analogous $\Phi_{\rho}$ in this case, we can certainly pullback  $\Phi_{\rho\vert_{\bmc_{\gamma}\bmH}}$ to $X^*(\bmH)$ but it depends on the choice of $\gamma$.
So we instead define in this case $\Phi_{\rho,\bmL}$ to be the set $\{\bmc_{\gamma}^* \Phi_{\rho\vert_{\bmc_{\gamma}\bmH}}\}$ as $\gamma$ ranges over $\bmX(\bmH,\bmL)\cap \Gamma$. We have
\begin{lem}\label{lemFinitePhi_Rho,L}
$\Phi_{\rho,\bmL}$ is a finite set.
\end{lem}

\begin{proof}
Indeed if not true, then there is a set $\Lambda$ in $\bmG(\Q)$ of bounded denominator yet $\rho\circ\bmc_{\gamma}(\Lambda)$ has unbounded denominator. This is a contradiction.
\end{proof}

For $\gamma \in \bmX(\bmH,\bmL)\cap \Gamma $ and a subset $\Phi$ of $\bmc_{\gamma}^* \Phi_{\rho\vert_{\bmc_{\gamma}\bmH}}$ we define

\begin{defi}
\begin{equation*}
\begin{aligned}
        \Omega(\bmc_{\gamma},\ep,\rho,\Phi):=&
    \{
    t\in \Lie(S_{\bmH}) \mid
    \inf_{0\neq v \in \bmV_{\alpha}(\Z)} ||(\gamma\exp{(t)}{\gamma}^{-1})v||\geq \ep, \,\forall \alpha \in (\bmc_{\gamma})_*\Phi
    \}\\
    =&
    \{
    t \in \Lie(S_{\bmH}) \mid
    \diff ({ \alpha\circ\bmc_{\gamma}})(t)\geq \ln\ep -\ln \inf_{0\neq v \in \bmV_{\alpha}(\Z)}||v||,\, \forall \alpha\in (\bmc_{\gamma})_*\Phi
    \}
\end{aligned}
\end{equation*} 
\end{defi}

If $\gamma$ is assumed to preserve the integral structure and $\bmL=\bmG$, then $\Omega(g,\ep,\rho,\Phi)=  \Omega(\bmc_{\gamma},\ep,\rho,\Phi)$.

In both situations we define
\begin{equation}\label{defPreImaPoly}
\begin{aligned}
    &\calP(g,\ep,\rho,\Phi):= \{h\in H \mid 
    \pi_{{}^{\circ}\bmH} (h) \in \exp\left(
     \Omega(g,\ep,\rho,\Phi)
    \right)
    \}  \\
    & \calP(\bmc_{\gamma},\ep,\rho,\Phi):= \{h\in H \mid 
    \pi_{{}^{\circ}\bmH} (h) \in \exp\left(
     \Omega(\bmc_{\gamma},\ep,\rho,\Phi)
    \right)
    \}. 
\end{aligned}
\end{equation}
Both are right invariant by ${}^{\circ}\bmH(\R)$ and hence by $\Gamma_{H}$ as $\Gamma_{H}$ is contained in ${}^{\circ}\bmH(\R)$.

\begin{prop}\label{exteriortrick}
Let $(\bmG,\bmH,\Gamma)$ be a standard triple.
Take a $\Q$-representation $(\rho,\bmV)$ of $\bmG$. Define $\rho'$ to be $\bigoplus_{i}\bigwedge^i\rho \big\vert_{\bmH}$. Take a non-empty open bounded subset $\calO\subset H$ and a positive number $\eta>0$.
Then there exists $\ep>0$ such that for all $g\in G$, 
$h\in \calP(g,\eta,\rho',\Phi_{\rho'})$ 
and $v_{\neq 0}\in \bmV(\Z)$, we have the inequality
\begin{equation*}
    \sup_{o\in \calO} ||
    g ho\cdot v
    ||\geq \ep.
\end{equation*}
\end{prop}

\begin{proof}
Consider 
\begin{equation*}
    \scrF:=\{
    f : \bmH \to \C \mid
    f(h)=\la hv,l\ra,\,\exists v\in \bmV ,\, l\in \bmV^*
    \}
\end{equation*}
where the angled bracket denotes the natural pairing between $\bmV$ and $\bmV^*$.
Then $\scrF$ is a finite-dimensional vector space. Because $\calO\cap \bmH(\Q)$ is Zariski-dense in $\bmH$, the natural map $\scrF \to \scrF\vert_{\calO\cap \bmH(\Q)}$ is an isomorphism. As this is finite dimensional, we may further find a finite subset $\Lambda\subset {\calO\cap \bmH(\Q)}$ such that $\scrF \to \scrF\vert_\Lambda$ is an isomorphism. This implies that $f\in \scrF$ vanishes on $\Lambda$ iff it vanishes on $\bmH$. We may find a positive integer $N$ such that $\Lambda \bmV(\Z) \subset \frac{1}{N}\bmV(\Z)$. Now we fix a non-zero $v\in \bmV(\Z)$.

Let $W$ be the $\Q$-linear subspace generated by $\bmH(\Q)\cdot v$. $W_{\C}$ is $\bmH$-invariant. For $A\subset \bmH(\Q)$, $A\cdot v$ spans $W$ iff all linear functionals $l$ that vanish on $A\cdot v$ also vanish on $W$. This is a condition on $\scrF$ and hence $\Lambda\cdot v$ spans $W$ and we may choose $\{\lambda_1,...,\lambda_k\}\subset \Lambda$ such that $\{\lambda_i v\}_i$ forms a basis of $W$. Let $w:=\wedge_i \lambda_i v$ then $N^kw \in \bigwedge^k\bmV(\Z)$. Also, $\bmH(\Q)$ preserves the line spanned by $w$. Hence by definition of $\Omega(g,\eta,\rho',\Phi_{\rho'})$, we have
\begin{equation*}
    N^k||gh (\wedge \lambda_i v)|| = ||g\pi_{{}^{\circ}\bmH}(h)(N^kw)|| \geq \eta.
\end{equation*}
But 
\begin{equation*}
    N^k||gh (\wedge \lambda_i v)|| \leq N^k \prod_i ||gh\lambda_i v|| 
    \leq N^k \sup_{o\in \calO} ||
    gho\cdot v
    ||^k,
\end{equation*}
therefore
\begin{equation*}
    \sup_{o\in \calO} ||
    gho\cdot v
    || \geq
    \eta^{1/k}N.
\end{equation*}
Taking $\ep:=\eta^{1/k}N$ completes the proof.
\end{proof}

Similarly we have
\begin{prop}\label{exteriortrick2}
Let $(\bmG,\bmH,\Gamma)$ be a standard triple and $\bmL$ be another $\Q$-subgroup of $\bmG$.
Take a $\Q$-representation $(\rho,\bmV)$ of $\bmL$. Define $\rho'$ to be $\bigoplus_{i}\bigwedge^i\rho$. Take a non-empty open bounded subset $\calO\subset H$ and a positive number $\eta>0$.
Then there exists $\ep>0$ such that for all $\gamma \in \bmX(\bmH,\bmL)$, 
$h\in \calP(\bmc_{\gamma},\eta,\rho',\bmc_{\gamma}^*\Phi_{\rho'\vert_{\bmc_{\gamma}\bmH}})$ 
and $v_{\neq 0}\in \bmV(\Z)$, we have the inequality
\begin{equation*}
    \sup_{o\in \calO} ||
    \bmc_{\gamma}(ho)\cdot v
    ||\geq \ep.
\end{equation*}
\end{prop}

\begin{proof}
Indeed the set of functions
\begin{equation*}
    \scrF:=\{
    f : \bmH \to \C \mid
    f(h)=\la \bmc_{\gamma} (h) v,l\ra,\,\exists v\in \bmV ,\, l\in \bmV^*,\, \gamma \in \bmX(\bmH,\bmL)\cap \Gamma
    \}
\end{equation*}
is also finite-dimensional.  The rest of the proof is almost identical as above.
\end{proof}

To relate this proposition with non-divergence, we need the notion of $(C,\alpha)$-good functions. Take a bounded open non-empty subset $\mathcal{D}$ in $\Lie(H)$. For a representation $(\rho,\bmV)$ of $\bmG$ and a pair $(v,l)\in \bmV \times \bmV^*$ and $g\in G$, define $ \phi_{g,v,l}: \mathcal{D}\to \C$ by
\begin{equation*}
    \phi_{g,v,l}(x):=\langle g\exp{(x)}v,l \rangle
\end{equation*} where the angled bracket denotes the natural pairing between $\bmV$ and $\bmV^*$.
Similarly,
if $\bmL$ is a connected $\Q$-subgroup of $\bmG$ and $(\rho,\bmV)$ is a representation of $\bmL$, 
for $g\in \bmX(\bmH,\bmL)$ and a pair $(v,l)\in \bmV\times \bmV^*$ we define $ \phi_{\bmc_{g},v,l}: \mathcal{D}\to \C$ by
\begin{equation*}
    \phi_{\bmc_{g},v,l}(x):=\langle \bmc_{g}(\exp{(x)}) v,l \rangle.
\end{equation*} 
Then there exist two positive numbers $C$ and $\alpha$ such that for all $(v,l)\in \bmV \times \bmV^*$ and $g\in G$, $\phi_{g,v,l}$ is $(C,\alpha)$\textbf{-good} which means that(see \cite[Section 3]{KleMar98})
\begin{equation*}
    \frac{1}{|B|} |
    \{
    x \in B \,\vert\, |\phi_{g,v,l}(x)|\leq \ep 
    \}
    | \leq  C
    (\frac{\ep}{\sup_{x\in B}|\phi_{g,v,l}(x)|})^{\alpha}
\end{equation*}
holds for all $\ep>0$ and open balls $B\subset \DD$.
Therefore $\phi_{g,v}(x):=||g\exp(x)v||$ is also $(C,\alpha)$-good on $\calD$ where we take $||\cdot||$ to be a sup-norm with respect to some basis of $\bmV$(see \cite[Lemma 3.1]{KleMar98}).

And similarly if we fix $\bmL$ and a representation $(\rho,\bmV)$ of $\bmL$, there exist two positive numbers $C$ and $\alpha$ such that for all $(v,l)\in \bmV \times \bmV^*$ and $g \in \bmX(\bmH,\bmL)$, $\phi_{\bmc_{g},v,l}$ is $(C,\alpha)${-good}. Also $\phi_{\bmc_g,v}(x):=||\bmc_g(\exp(x))v||$ is also $(C,\alpha)$-good.

In both cases the set of functions $ \phi_{g,v,l}$ and $\phi_{\bmc_{g},v,l}$ span a finite-dimensional space of analytic functions on $\calD$. Hence \cite[Proposition 3.4]{KleMar98} implies that these two collections of functions are $(C,\alpha)$-good for some $C$, $\alpha$ positive.

We also need a qualitative version of a theorem of Kleinbock--Margulis \cite[Theorem 5.2]{KleMar98}. We have implicitly chosen a sup-norm in the representation space.
\begin{thm}\label{kleinbockMargulis}
Given a linear algebraic group $\bmG$ and connected $\Q$-subgroups $\bmH$ and $\bmL$. 
Let $\mathcal{D}$ be a non-empty open bounded subset in $\Lie(H)$. 
Take a representation $(\rho_1,\bmV_1)$(resp. $(\rho_2,\bmV_2)$) of $\bmG$(resp. $\bmL$). 
We fix an integral structure on $\bmV_1$(resp. $\bmV_2$).
There exists a constant $C'>0,\alpha>0$ and $0<\eta<\frac{1}{\dim\bmV_i}$($i=1$ or $2$) such that the following is true.
For each ball $B$ such that $3^{\dim\bmV_1}B\subset \mathcal{D}$ and $g\in G$ satisfying that
\begin{equation*}
\begin{aligned}
    \sup_{x \in B}||g\exp{(x)}v|| \geq \eta \quad \forall v_{\neq 0} \text{ pure wedge in } \bigwedge^{i}\bmV_1(\Z),\, \forall i,
\end{aligned}
\end{equation*}
we have whenever $\ep\leq \eta$,
\begin{equation*}
    \frac{1}{|B|} |\{x \in B \,\vert\, \inf_{0\neq v\in \bmV_1(\Z)} ||g\exp{(x)}v|| \geq \ep\}|
    \leq C'(\frac{\ep}{\eta})^{\alpha}.
\end{equation*}
Similarly, for all ball $B$ such that $3^{\dim\bmV_2}B\subset \mathcal{D}$ and $g\in\bmX(\bmH,\bmL)$ satisfying that
\begin{equation*}
     \sup_{x \in B}||\bmc_{g}(\exp{(x)})v|| \geq \eta \quad \forall v_{\neq 0} \text{ pure wedge in } \bigwedge^{i}\bmV_2(\Z),\, \forall i,
\end{equation*}
we have whenever $\ep\leq \eta$,
\begin{equation*}
    \frac{1}{|B|} |\{x \in B \,\vert\, \inf_{0\neq v\in \bmV_2(\Z)} ||\bmc_{g}\exp{(x)}v|| \geq \ep\}|
    \leq C'(\frac{\ep}{\eta})^{\alpha}.
\end{equation*}
\end{thm}

For a $\Q$-representation $(\rho,\bmV)$ of $\bmG$ (resp. $\bmL$) in $\SL(\bmV)$, to transfer non-divergence on $\SL(V)/\SL(\bmV_\Z)$ back to $G/\Gamma$ (resp. $L/\Gamma_L$), we need $\rho(\bmG)$ (resp. $\rho(\bmL)$) to be an observable subgroup of $\SL(\bmV)$.
Recall that a representation $\rho:\bmG\to \SL_N $ is said to be {observable} if its image is an observable subgroup of $\SL_N$. Let $\bmG_m$ be the linear algebraic group over $\Q$ whose $\Q$-points are $\Q^{\times}$. We record here a useful lemma.

\begin{lem}\label{tensortrick}
Let  $\{\bmV_i\}_{i=1,...,l}$ be vector spaces. 
$\bmG^l_m$ naturally acts on $\prod_i (\bmV_i \bs \{0\})$ and $\bigotimes \bmV_i$. 
Then the natural map 
$\left(\prod_i (\bmV_i \bs \{0\})\right) /\bmG^l_m \to  (\bigotimes \bmV_i)/\bmG^l_m$ is injective.
\end{lem}

Proof is omitted.

\begin{lem}\label{lemObsRepExi}
For each $\Q$-representation $\rho:\bmG\to \SL_n$ there exists an observable $\Q$-representation 
 $\rho':\bmG\to \SL_N $ containing $\rho$ as a direct summand.
\end{lem}

\begin{proof}
By Chevalley‘s lemma(see \cite[Lemma 5.5.1]{Spr98}), there is a representation $\psi:\SL_n \to \SL(\bmV)$ and a non-zero $\Q$-vector $v$ such that $g\in \SL_n$ stabilize the line $[v]$ spanned by $v$ iff $g$ is contained in $\rho(\bmG)$. Then there is a character $\alpha : \bmG \to \bmG_m$ such that $\rho(g)v=\alpha(g)v$. Take $N=n +2$ and let $\rho'(g):=
\left[
\begin{array}{c|c}
     \rho(g)
 & 
 \\
\hline 
    & 
    \begin{array}{cc}
         \alpha(g)   &  \\
         & \alpha^{-1}(g)
    \end{array}
\end{array}
\right]
$. It is clear that the image of $\rho'$ lands in $\SL_{N}$ and $\rho$ is a direct summand of $\rho'$. We claim that $\rho'$ is observable.

Embed $\SL_n$ into $\SL_N$ in the upper-left corner and define
$\bmF=\left[
\begin{array}{c|c}
     \SL_n
 & 
 \\
\hline 
    & 
    \begin{array}{cc}
         t  &  \\
         &  t^{-1}
    \end{array}
\end{array}
\right]
$. First we extend $\psi$ to a representation of $\bmF$ such that $\left[
\begin{array}{c|c}
     I_n
 & 
 \\
\hline 
    & 
    \begin{array}{cc}
         t  &  \\
         &  t^{-1}
    \end{array}
\end{array}
\right]$ acts as identity.

As $\bmF$ is observable in $\SL_N$, we may take a representation $(\widetilde{\psi}_1, \widetilde{\bmV}_1)$ of $\SL_N$ whose restriction to $\bmF$ contains $\psi$ as a direct summand by item (3) in Definition \ref{defiObs}. In particular there exists a non-zero $\Q$-vector $\widetilde{v}_1$ such that $g\in \SL_n$ stabilize $[\widetilde{v}_1]$ iff $g$ is contained in $\rho(\bmG)$ and $\rho(g)$ acts by $\alpha(g)$.
Let $(\widetilde{\psi}_2, \widetilde{\bmV}_2)$ be the standard representation of $\SL_N$ and $\widetilde{v}_2$ to be $e_{N}=e_{n+2}$.
Then $\rho'(\bmG)$ fixes the vector $\widetilde{v}_1\otimes\widetilde{v}_2$. Moreover by Lemma \ref{tensortrick} above, $g\in\bmF$ fixes $\widetilde{v}_1\otimes\widetilde{v}_2$ iff $g$ is contained in $\rho'(\bmG)$.

Now take $\widev_3:=e_{1}\wedge...\wedge e_n$ and $\widev_4:=e_{n+1}\otimes e_{n+2}$. Then $g\in\SL_N$ fixes $\widev_3\oplus\widev_4$ iff $g$ is contained in $\bmF$. Hence $\rho'(\bmG)$ is exactly the stabilizer of $\widev_1 \otimes \widev_2 \oplus \widev_3 \oplus \widev_4$. So it follows that $\rho'(\bmG)$ and $\rho'$ are observable.
\end{proof}

Given a faithful observable $\Q$-representation, assuming $\rho(\bmG(\Z))\subset \SL_N(\Z)$, the induced map $G / \bmG(\Z) \to \SL_N(\R)/\SL_N(\Z)$ is a proper map. Combined with Mahler's criterion, we have the following lemma:

\begin{lem}\label{mahler}
Let $\bmG$ be a linear algebraic group over $\Q$ and $\Gamma\leq \bmG(\Q)$ be commensurable with $\bmG(\Z)$. Take a faithful observable $\Q$-representation $\rho: \bmG\to \SL(\bmV)$ and a lattice $\bmV(\Z)\subset \bmV(\Q)$ that is preserved by $\Gamma$. Define for each $\ep>0$,
\begin{equation*}
    K_{\eta}(\rho):=\{ \pi(g)\in G/\Gamma \, \vert\, 
    \inf_{0\neq v \in \bmV(\Z)}||gv|| \geq \eta
    \}.
\end{equation*}
Then as $\eta$ decreases to $0$, $\{K_{\eta}(\rho)\}$ forms an increasing family of compact sets whose interiors cover $G/\Gamma$.
\end{lem}

The Proposition \ref{exteriortrick}, Theorem \ref{kleinbockMargulis} and Lemma \ref{mahler} above together imply that 
\begin{prop}\label{nondiverg}
Same notation as in Proposition \ref{exteriortrick} and we further assume that the $\Q$-representation $(\rho,\bmV)$ is faithful into $\SL(\bmV)$ and observable.
Let $\rho''$ be a further exterior product $\bigoplus_{i}\bigwedge^i\rho'$. 
Then for any $\delta >0$ and $\eta>0$ there exist $\ep>0$  such that 
for all $g\in G$, $h\in \calP(g,\eta,\rho'',\Phi_{\rho''})$  we have
\begin{equation*}
    \widehat{\mu}_{\calO}\{\pi_{\Gamma}(o)
    \,\vert\,
    \pi_{\Gamma}(gho) \in K_{\ep}(\rho)
    \} \geq 1-\delta.
\end{equation*}
In particular, for any $\eta>0$, $\{g_n\}\subset G$ and $h_n\in \calP(g_n,\eta,\rho'',\Phi_{\rho''})$, all weak-$*$ limits of $(g_nh_n)_{*}\widehat{\mu}_{\calO}$ are probability measures.
\end{prop}

Similarly,
\begin{prop}\label{nondiverg2}
Same notation as in Proposition \ref{exteriortrick2} and we further assume that the $\Q$-representation $(\rho,\bmV)$ is faithful into $\SL(\bmV)$ and observable.
Let $\rho''$ be a further exterior product $\bigoplus_{i}\bigwedge^i\rho'$. 
Then for any $\delta >0$ and $\eta>0$ there exists $\ep>0$ such that 
for all $\gamma \in \bmX(\bmH,\bmL)\cap \Gamma$,
$h\in \calP(\bmc_{\gamma},\eta,\rho'',\bmc_{\gamma}^*\Phi_{\rho''\vert_{\bmc_{\gamma}\bmH}})$
we have
\begin{equation*}
    \widehat{\mu}_{\calO}\{\pi_{\Gamma}(o)
    \mid
    \pi_{\Gamma}(\bmc_{\gamma}(ho)) \in K_{\ep}
    \} \geq 1-\delta.
\end{equation*}
In particular, for $\eta>0$, a sequence $\{\gamma_n\}$ in  $\bmX(\bmH,\bmL)\cap \Gamma$ such that 
all $\bmc_{\gamma_n}^*\Phi_{\rho''\vert_{\bmc_{\gamma_n}\bmH}}$'s are equal to the same $\Phi$
and a sequence $\{h_n\}$ with $h_n$ in $\calP(\bmc_{\gamma_n},\eta,\rho'',\Phi)$ for each $n$, 
all weak-$*$ limits of $(\gamma_nh_n)_{*}\widehat{\mu}_{\calO}$ are probability measures.
\end{prop}

The converse of Proposition \ref{nondiverg}, \ref{nondiverg2} is also true:
\begin{lem}
For $\ep>0$, a non-empty open bounded set $\calO\subset H$, a $\Q$-representation $\rho_1:\bmG\to\SL(\bmV)$ and another faithful observable $\Q$-representation $\rho_2$ defining $K_{\ep}(\rho_2)$, there exists $\eta>0$ such that for all $g\in G$ and $h\in H$ satisfying
\begin{equation*}
    \pi_{\Gamma}(gh\calO) \cap K_{\ep}(\rho_2) \neq \emptyset,
\end{equation*}
we have $h\in \calP(g,\eta,\rho,\Phi_{\rho})$. Similar statements hold in the other situation.
\end{lem}

\begin{defi}
We say a $\Q$-representation $\rho:\bmG \to \SL_N$ is \textbf{superfaithful} iff it contains all double exterior products of a faithful observable $\Q$-representation into $\SL_n$ as a direct summand.
\end{defi}

Such a representation always exists thanks to Lemma \ref{lemObsRepExi} and Theorem \ref{thmNondiv} follows from Proposition \ref{nondiverg} above. Note that a superfaithful representation is still faithful and observable.

\subsection{Case of maximal split tori}\label{secNondivMaxSplTori}

Let $\bmG$ be a semisimple group over $\Q$ and $\bmH$ be a maximal $\Q$-split torus in $\bmG$. Let $\calW$ be the $\Q$-Weyl group associated with $\bmH$. Fix a set of representative $\dot{\calW}=\{w\}$ of $\calW$ in $\bmG(\Q)$. Also fix a $\Z$-structure on $\frakg$ which is preserved by $\Ad(\Gamma)$. For each $\Q$-parabolic subgroup $\bmP$, let $v_{\bmP}$ be the unique up to sign primitive vector in $\wedge^{\dim \bmP}\frakg_{\Z}$ that represents the Lie algebra of $\bmP$. 
Note that $\gamma \cdot v_{\bmP} = v_{\gamma \bmP \gamma^{-1}}$.
For $g\in G$, define $d_{\bmP}(g):=||g\cdot v_{\bmP} ||$ where the $G$-action is induced from the Adjoint action. Analogous to Mahler's criterion, by using reduction theory, it can be shown that:
\begin{prop}
Define for $\eta >0$,
\begin{equation*}
    K_{\eta}:=\{[g]\in G/\Gamma\,\vert\,
d_{\bmP}(g) \geq \eta ,\,\text{ for all maximal proper }\Q\text{-parabolic subgroups } \bmP
\}.
\end{equation*}
Then as $\eta>0$ decreases to $0$, $\{K_{\eta}\}$ forms a family of compact sets. And the union of their interiors covers $G/\Gamma$.
\end{prop}

Let $\scrP$ denote the  set of maximal proper $\Q$-parabolic subgroups of $\bmG$ and $\scrP_{\bmH}$ denote those containing $\bmH$. In the present case, $\scrP_{\bmH}$  is a finite set. 
For $g\in G$, $\eta>0$, define 
\begin{equation}\label{EquPolyMaxTori}
    \Omega_{g,\ep}:=\{
    t\in \Lie(H)\,\vert\,
    d_{\bmP}(g\exp{(t)}) \geq \ep,\, \forall \bmP \in \scrP_{\bmH}
    \}
\end{equation}
which is a bounded convex polytope with finitely many sides. Note that $ d_{\bmP}(g\exp{(t)})$ is equal to $\exp(\diffalpha_{\bmP}(t))d_{\bmP}(g)$ for some character $\alpha_{\bmP}$ depending on $\bmP$.

The following has been shown in \cite{EskMozSha97}. A different proof can be obtained by modifying the argument in \cite{KleMar98}.
\begin{thm}
Fix a nonempty open bounded subset $\calO$ of $H$ and two positive numbers $\ep$ and $\delta$. 
Then there exists another positive number $\eta$ such that
for all $g\in G$ that satisfies
\begin{equation*}
    \sup_{o\in\calO}||gov_{\bmP}|| > \ep,\quad \forall \bmP\in\scrP,
\end{equation*}
we have 
\begin{equation*}
    \whmu_{\calO} \{o \in \calO\,\vert \, 
    \pi_{\Gamma}(go) \notin K_{\eta}
    \} < \delta.
\end{equation*}
\end{thm}

We are going to show the assumption is satisfied. 

\begin{prop}\label{propNondivMaxTori}
Fix a nonempty open bounded subset $\calO$ of $H$ and a positive number $\ep$.
Then there exists another $\ep'>0$ such that for all $t \in \Omega_{g,\ep}$, 
\begin{equation*}
    \sup_{o\in\calO}||g\exp{(t)}ov_{\bmP}|| > \ep',\quad \forall \bmP\in\scrP.
\end{equation*}
\end{prop}

Hence whenever $\bmH$ is a maximal $\Q$-split torus in a semisimple group $\bmG$, one can replace the $\Omega$'s in the next section by this one and prove similar statements. This will be needed for Theorem \ref{thmCountSym}.

\begin{conj}
The same thing is true for any other $\Q$-subgroup $\bmH$ where $\Omega_{g,\ep}$ is defined to be a subset of the Lie algebra of $S_{\bmH}$.
\end{conj}

One can give a proof of this when $\bmG=\SL_N$. As this is neither used in the main theorem nor in the applications, we shall omit the proof here.

A proof of this conjecture should yield results on equidistribution on the boundary in light of the work \cite{DawGoroUll18} and \cite{DawGoroUllLi19}.

Before we start the proof of Proposition \ref{propNondivMaxTori}, we need \cite[Lemma 4.1]{EskMozSha97} adapted to our situation.

\begin{lem}
For a set of distinct characters $\{\alpha_1,...,\alpha_n\}$ on $\bmH$ and a nonempty open set $\calO$ in $H$. Fix a vector space and a norm $||\cdot||$. Then there exists $\kappa>0$ such that for any $n$ vectors $v_1,...,v_n$, one has
\begin{equation*}
    \sup_{o\in\calO} ||\sum_i \alpha_i(o) v_i|| \geq \kappa \sup_i ||v_i||.
\end{equation*}
\end{lem}

\begin{proof}[Proof of Proposition \ref{propNondivMaxTori}]
By reduction theory, there exists $N_0 \in \Z^+$ such that for all $\bmP \in \scrP$, there exists $\gamma_{\bmP} \in \bmG(\Q)$ and $\bmP' \in \scrP_{\bmH}$ such that 
\begin{equation*}
    v_{\bmP} = {\lambda}\gamma_{\bmP}\cdot v_{\bmP'} 
\end{equation*}
for some $\frac{1}{N_0}\leq |\lambda| \leq N_0$. 
For each character $\alpha$ of $\bmH$ appearing in the Adjoint representation and all its exterior powers, define $\pi_{\alpha}$ to be the projection to the corresponding weight space. After enlarging $N_0$, we may assume that $\pi_{\alpha}$ maps each integral vector to some vector that is either $0$ or has norm at least $\frac{1}{N_0}$. We also apply the Lemma above to the Adjoint representation and all its exterior powers and all weights that appear here to get some $\kappa>0$. 

Now fix such a $\bmP_0$, write $\gamma_0$ for $\gamma_{\bmP_0}$ and write $\lambda_0$ for the $\lambda$ appearing here. 
Take $\bmP_1'$ to be a minimal parabolic subgroup that is contained in $\bmP'_0$ and contains $\bmH$. Then $\bmP_1'=\bmM'_1\cdot \bmH \cdot \bmU_1'$ where $\bmU_1'$ is the unipotent radical of $\bmP_1'$ and $\bmM'_1\leq \bmZ_{\bmH}\bmG$ is $\Q$-anisotropic. Note that $\bmM'_1$ fixes $v_{\bmP}$ for all $\bmP$ containing $\bmH$ and preserves each weight space of $\bmH$.

By Bruhat decomposition, one can write $\gamma_{0}= z_0 u_0 w_0 p_0$ for some $z_0 \in \bmM'_1(\Q)$, $u_0 \in \bmU_1'(\Q)$, $w_0 \in \dot{\calW}$ and $p_0 \in \bmP_1'(\Q)$. Hence 
\begin{equation*}
    \gamma_0 v_{\bmP_0'} = a_1 (z_0 u_0)\cdot v_{w_0\bmP_0'w_0^{-1}} = a_0 v_{w_0\bmP_0'w_0^{-1}} + \sum v_{\alpha}
\end{equation*}
for some $a_0, a_1\neq 0$ and some weight vectors $v_{\alpha}$ with respect to certain character $\alpha$ that is \textit{distinct} from the one corresponding to $w_0\bmP_0'w_0^{-1}$. To see why $v_{\alpha}$'s have different weights, one may write $u_0$ as the exponential of some nilpotent element whose action would always change the weight. And remember that $z_0$ preserves each weight space. Thus if $\alpha_0$ is the weight corresponding to $w_0\bmP_0'w_0^{-1}$, then 
\begin{equation*}
    a_0 v_{w_0\bmP_0'w_0^{-1}}=\pi_{\alpha_0}(\gamma_0 v_{\bmP_0'} )= \pi_{\alpha_0}(\lambda_0^{-1} v_{\bmP_0}),
\end{equation*}
implying $|a_0| \geq \frac{1}{N_0^2}$. Now take $t\in \Omega_{g,\ep}$,
\begin{equation*}
\begin{aligned}
          &\sup_{o\in\calO}||g\exp{(t)}ov_{\bmP_0}|| \\
          =&   \sup_{o\in\calO}||\lambda_0 g\exp{(t)} o \gamma_0v_{\bmP'_0}|| \\
          = & \sup_{o\in\calO}||\lambda_0 g\exp{(t)} o (a_0 v_{w_0\bmP_0'w_0^{-1}} + \sum_{\alpha\neq \alpha_0} v_{\alpha})|| \\
          =& 
          |\lambda_0| \sup_{o\in\calO}
          ||g \exp{(t)} (\alpha_0 (o)  a_0v_{w_0\bmP_0'w_0^{-1}}+ \sum_{\alpha\neq \alpha_0} \alpha (o) v_{\alpha})|| \\
          =& 
          |\lambda_0|\sup_{o\in\calO} 
          ||
          \alpha_0 (o) g \exp{(t)}  a_0 v_{w_0\bmP_0'w_0^{-1}} +
           \sum_{\alpha\neq \alpha_0} \alpha (o) g \exp{(t)} v_{\alpha})
          ||\\
          \geq&
          \kappa|\lambda_0| \cdot ||g \exp{(t)} a_0 v_{w_0\bmP_0'w_0^{-1}}||\\
          \geq& 
          \kappa |\lambda_0| \ep |a_0| \geq \frac{1}{N_0^3}\kappa\ep .
\end{aligned}
\end{equation*}
Setting $\ep':=\frac{1}{N_0^3}\kappa\ep $ then concludes the proof.
\end{proof}

\section{Translates by $\Gamma$ and equidistribution}\label{secGammaEqui}

In this section we are given an observable standard triple $(\bmG,\bmH,\Gamma)$ and a sequence $\{\gamma_n\}\subset \Gamma$ and we wish to study the possible limits of $(\gamma_n)_*\mu_{H}$.

\begin{defi}\label{defminimal}
Given a standard triple $(\bmG,\bmH,\Gamma)$, a connected $\Q$-subgroup $\bmL\leq \bmG$ and a sequence $\{\gamma_n\}\subset \Gamma$.
We say that $(\{\gamma_n\},\bmL)$ is \textbf{minimal} for $\bmH$ if and only if for all infinite subsequences $\{n_k\}$, the closed subgroup generated by
$\bigcup \gamma_{n_k}\bmH\gamma_{n_k}^{-1}$ is equal to $\bmL$
 and  \textbf{potentially minimal} for $\bmH$ if and only if for all infinite subsequences $\{n_k\}$, the closed subgroup generated by $\bigcup \gamma_{n_k}\bmH\gamma_{n_k}^{-1}$ is epimorphic in $\bmL$.
\end{defi}

Recall that a subgroup $\bmH$ of $\bmG$ is said to be \textbf{epimorphic} iff for every representation $(\rho,\bmV)$ of $\bmG$ and every $v\in \bmV$ that is fixed by $\bmH$, $v$ is also fixed by $\bmG$. The notion of being epimorphic is closely related to being observable. Take $\bmH$ to be a subgroup of $\bmG$ and $\bmL$ to be the smallest observable subgroup of $\bmG$ containing $\bmH$. Then $\bmH$ is epimorphic in $\bmL$. And if $\bmH$ is epimorphic in another subgroup $\bmF$ of $\bmG$, then $\bmF$ is contained in $\bmL$. This $\bmL$ is called the \textbf{observable hull} of $\bmH$ in $\bmG$.

  We need the following important input from the work of Eskin--Mozes--Shah 
  (see \cite[Theorem 2.1]{EskMozSha96} and \cite{EskMozSha96Errat}).  Note that a connected real algebraic group for them is $G$ here for some linear algebraic group $\bmG$ defined over $\R$.
    
    \begin{thm}\label{EMSTheorem}
    Let $(\bmG,\bmH,\Gamma)$ be a standard triple, $\bmL$ be a connected $\Q$-subgroup and $\calO\subset H$ be a non-empty open bounded subset.
    Assume that we are given a sequence of morphisms 
    $\{\bmc_i:\bmH\to \bmL \}_{i\in \Z^+}$ of algebraic groups over $\Q$ such that 
    \begin{itemize}\label{assumptionEMS}
        \item[(1)] no proper $\Q$-subgroup of $\bmL$ contains $\bmc_i(\bmH)$ for infinitely many $i$;
        \item[(2)] for every $h\in \bmH(\Q)$, there exists $k\in \Z^+$ such that $\{\bmc_i(h)\}\subset \bmL(\frac{1}{k}\Z)$;
        \item[(3)] for each sequence $\{h_i\}$ in $H$ that converges to $e_{\bmH}$, all the eigenvalues for the action of $\Ad(\bmc_i(h_i))$ on $\Lie(\bmL)$ tend to $1$ as $i\to\infty$;
        \item[(4)] for each regular algebraic function $f$ on $\bmL$, $\{ \bmc_i^*(f)\}$ spans a finite-dimensional space of functions on $\bmH$;
        \item[(5)] for all $i$, $\bmc_i(\Gamma_H) \subset \Gamma_L$.
    \end{itemize}
    By the last assumption, $\bmc_i$ induces a map $H/\Gamma_H \to L/\Gamma_L$.
    Take $\nu$ to be any limit point of $(\bmc_i)_*\widehat{\mu}_{\calO}$,
     then $\nu$ is an ${}^{\circ}L$-invariant probability measure.
    \end{thm}
    
    The only difference from the original statement is that neither $\bmH$ nor $\bmL$ is assumed to have no non-trivial $\Q$-characters and a full homogeneous measure is replaced by $\mu_\calO$.  Let me briefly recall the proof to assure the reader that almost the same proof still works. Undefined terms can be found in \cite{EskMozSha96}. 
    
    \begin{enumerate}
        \item One may assume that $\calO=h_0\exp(B)$ for some $h_0\in \bmH(\Q)\cap H$ and $B \subset \Lie(H)$ open bounded 
        and that by $(C,\alpha)$-goodness $\nu$ is a probability measure. 
        In the original exposition of \cite{EskMozSha96} another class of functions $E_G(m,n,\Lambda)$ is used. Let us briefly explain we indeed have the $(C,\alpha)$-goodness. Fix a finite-dimensional 
        $\R$-representation $V$ of $\bmG$, let
        \begin{equation*}
            \scrF:=\{f:B\to \C\,\vert\, 
            f(x)=\left\la
            g\bmc_i(h'\exp(x))v,l
            \right\ra,\,v\in V,\,l\in V^*,\,h'\in H,\,i\in\Z^+
            \}
        \end{equation*}
        Then $\scrF$ spans a finite-dimensional vector space $\calF$. This is because the matrix coefficients are finite-dimensional, the fourth property from Theorem \ref{assumptionEMS} holds and for any finite-dimensional subspace of regular functions on $\bmH$ there exists a bigger one that is still of finite dimension and left $\bmH$-invariant(see \cite[Proposition 2.3.6]{Spr98}). Moreover, each function in $\calF$ is analytic and hence one may conclude with \cite[Proposition 3.4]{KleMar98}.
        \item Next one shows that $v$ is invariant under a non-trivial one-parameter unipotent subgroup $U$. 
        Indeed having no non-trivial $\Q$-characters plays no role in the proof of \cite[Proposition 2.2]{EskMozSha96}. Invoking rigidity theorem one obtains a $\Q$-subgroup $\bmF$ with $\nu(\pi_{\Gamma}S(\bmF,U))=0$ and $\nu(\pi_{\Gamma}N(\bmF,U))\neq 0$.
        \item Then one applies \cite[Proposition 3.13]{EskMozSha96}. Though stated for $E_G(m,n,\Lambda)$ it holds also for $(C,\alpha)$-good functions.  Then one may continue with the argument on the last four paragraphs of \cite[Page 273]{EskMozSha96} and note that the argument on page 274 therein has been replaced by \cite{EskMozSha96Errat}. This part of arguments also makes no use of having no non-trivial $\Q$-characters so it carries through.
        \item Now we obtain a normal $\Q$-subgroup $\bm{F}$, possibly different from the $\bm{F}$ above, that is of Ratner class and contains $U$ above such that $\nu$ is $F$-invariant. Look at the fibre bundle
        \begin{equation*}
             \begin{tikzcd}
         Fg\Gamma/\Gamma= g F\Gamma/\Gamma \arrow[r, hook]
         & \displaystyle\quotient{G}{\Gamma}  \arrow[d, "\pi_F"]\\
          & \displaystyle{
          \quotient{\left(G/{F}\right)
          }
          {\pi_F(\Gamma)}
          }
         \end{tikzcd}.
        \end{equation*}
       One may decompose the measure \begin{equation*}
           \nu=\int_{x\in {\left({G}/{F}\right)
          }\big/
          {\pi_F(\Gamma)}
          }
          \widehat{\mu}_{Fx\Gamma}
          {(\pi_F)}_*\nu(x)
       \end{equation*}
       where $\widehat{\mu}_{Fx\Gamma}$ is the unique $F$-invariant probability measure supported on $Fx\Gamma/\Gamma$.  By induction (one can verify the natural induction hypothesis is satisfied) we see that $ {(\pi_F)}_*\nu$ is invariant by
       ${}^{\circ}\left(
       G/F
       \right) = {}^{\circ}G/F$. Then one can check by the integral expression above that $\nu$ is ${}^{\circ}G$-invariant.
    \end{enumerate}

\begin{coro}\label{coroEMS}
Suppose that we are given a standard triple $(\bmG,\bmH,\Gamma)$, a sequence $\{\gamma_n\}\subset \Gamma$ and a connected $\Q$-subgroup $\bmL \leq \bmG$. Assume that $(\{\gamma_n\},\bmL)$ is minimal for $\bmH$.
Consider the map $\bmc_n: H/\Gamma_H \to L/\Gamma_L$ induced from $\bmc_{\gamma_n}$ defined by $\bmc_{\gamma_n}(h)=\gamma_nh\gamma_n^{-1}$.
Then for all non-empty open bounded subsets $\calO$ of $H$, all weak-$*$ limits of $\{(\bmc_n)_*\widehat{\mu}_{\calO}\}$ in $L/\Gamma_L$ are ${}^{\circ}L$-invariant probability measures.
\end{coro}

Note that for $N\subset M$ a closed embedded submanifold, a sequence of probability measures $\{\mu_n\}$ that converges to $\mu$ on $N$ also converges to $\mu$ on $M$. Hence in order for us to understand the limit of $\{(\bmc_n)_*{\mu}_{H}\}$, in light of the above corollary, two things remain to be done:
\begin{itemize}
    \item merely being ${}^{\circ}L$-invariant is not satisfactory. We want $\{(\bmc_n)_*{\mu}_{H}\}$ to converge to $\mu_L$ on $L/\Gamma_L$;
    \item find such a $\bmL$ that is observable in $\bmG$.
\end{itemize}

We shall take care of the first point in Section \ref{sec3PartI} and the second point in Section \ref{sec3PartII}.
Before that we shall give some preliminaries on polytopes in Section \ref{secPolytopes}.

\subsection{Polytopes and Cones}\label{secPolytopes}

\begin{defi}\label{defiboundedFunctionals}
Let $(\bmG,\bmH,\Gamma)$ be a standard triple and $\bmL$ be a connected $\Q$-subgroup of $\bmG$.
Given a $\Q$-representation $\rho:\bmL \to\SL_N(\bmV)$ and a sequence $\{\gamma_n\}$ such that $\bmc_{\gamma_n} (\bmH) $ is contained in $\bmL$ for all $n$. 
For an element $\Phi_{\rho}$ in $\Phi_{\rho,\bmL}$ and a subset $\Phi$ of $\Phi_{\rho}$,
we say that $\{\gamma_n\}$ is 
$\Phi$\textbf{-clean} iff 
$\bmc_{\gamma_n}^* \Phi_{\rho\vert_{\bmc_{\gamma_n}\bmH}}$ is equal to $\Phi_{\rho}$ for all $n$ and
for each $\alpha\in\Phi$, denoting 
$\alpha_n := (\bmc_{\gamma_n})_* \alpha \in X^*(\bmc_{\gamma_n}(\bmH))$, either
\begin{enumerate}
    \item $\inf_{0\neq v\in \bmV_{\alpha_n}(\Z)} || v|| \to +\infty$, or
    \item $\inf_{0\neq v\in \bmV_{\alpha_n}(\Z)} ||v|| $ remains bounded.
\end{enumerate}
\end{defi}

Recall that by Lemma \ref{lemFinitePhi_Rho,L}, $\Phi_{\rho,\bmL}$ is a finite set.
Hence for arbitrary sequence $\{\gamma_n\}\subset \bmX(\bmH,\bmL)\cap \Gamma$, by passing to a subsequence, we may always assume that $\{\gamma_n\}$ is $\Phi_{\rho}$-clean for some $\Phi_{\rho}\in \Phi_{\rho,\bmL}$.

\begin{defi}\label{defiofphi}
For a $\Phi$-clean sequence $\{\gamma_n\}$, let $\Phi_{\infty}(\{\gamma_n\}) \subset \Phi$ consist of those $\alpha$'s that fall in case $(1)$ and $\Phi_{bdd}(\{\gamma_n\}) $ be its complement. 
We also define 
\begin{equation*}
\begin{aligned}
    \Phi_0(\{\gamma_n\}):=& \{
    \alpha\in \Phi_{bdd}(\{\gamma_n\})\,\vert\,
    \exists \alpha \in I \subset \Phi_{bdd},\, 
    \exists \{a_{\beta}\}_{\beta \in I} \subset \R_{>0},\, 
    \sum_{\beta\in I} a_{\beta}\beta=0
    \} \\
    \Phi_1(\{\gamma_n\}):=& \,\Phi_{bdd}(\{\gamma_n\}) \bs \Phi_0(\{\gamma_n\})
\end{aligned}
\end{equation*}
\end{defi}
Note that by definition, $\Phi=\Phi_{\infty}(\{\gamma_n\})\sqcup \Phi_1(\{\gamma_n\})\sqcup\Phi_0(\{\gamma_n\})$.

Though in the definition of $\Phi_0$, $a_{\beta}$'s are just positive real numbers, but as characters form a $\Z$-lattice in its $\R$-linear span we may and do choose $\{a_{\beta}\}$ to be positive rational numbers and even positive integers.

Now we turn to some generalities on polytopes and cones.
For $V$ a finite-dimensional $\R$-vector space and $\Phi\subset V^*$ a finite collection of functionals, define
\begin{equation*}
    \Cone(\Phi):=\{v\in V  \,\vert\, \alpha(v) \geq 0,\, \forall \alpha\in \Phi\}.
\end{equation*} 
Let $\Phi_0$ be defined the same way as in Definition \ref{defiofphi} 
replacing $\Phi_{bdd}$ by $\Phi$. 
\begin{defi}
Define $W(\Phi)$ to be the $\R$-linear subspace spanned by $\Cone(\Phi)$ and $\pi_{W(\Phi)}:V\to V/W(\Phi)$ to be the natural projection. Moreover if V is equipped with an Euclidean metric, we identify $V/W(\Phi)$ as the orthogonal complement $U(\Phi)$ of $W(\Phi)$ in $V$.
\end{defi}
It is clear that $\Cone(\Phi)=\Cone(\Phi)\cap W(\Phi)$ is open in $W(\Phi)$. Let $\Phi'_0:=\{\alpha\in \Phi,\, \alpha\vert_{W(\Phi)} =0 \}$.

\begin{lem}\label{lemmaStrucCone}
$\Phi_0=\Phi_0'$ and $W(\Phi) =\ker \Phi_0$.
\end{lem}

\begin{proof}
We write $W=W(\Phi)$ in the proof.
First note that there exists $v\in \Cone(\Phi)$ such that $\alpha(v)>0$ for all 
$\alpha \in \Phi\setminus \Phi'_{0}$. Let us fix such a vector and name it $v_0$.
Indeed, the weaker claim allowing $v$ to depend on $\alpha \in \Phi\setminus \Phi'_{0}$ is obvious. 
Then one may just sum them together.

For $w \in \Cone(\Phi'_{0})$ small enough, $v_0+w$ is still contained in $\Cone(\Phi)\subset W$ and hence $w$ is in $W$.  So $\Cone(\Phi'_{0})=W=\ker(\Phi'_{0})$. It remains to show $\Phi_0=\Phi'_0$.

The non-trivial direction is $\Phi'_{0} \subset \Phi_0$. We first show that $0\in V^*$ is in the interior of the cone spanned by $\Phi'_{0}$ relative to the subspace spanned by the cone. 
Indeed, if not true, then $0\in V^*$ is in the relative boundary of the cone spanned by $\Phi'_{0}$.
By Hahn--Banach theorem, there exists $v\in V$ such that $\alpha(v)\geq 0$ for all $\alpha \in \Phi'_{0}$ and there exists $l$ in the cone spanned by $\Phi'_{0}$ such that $l(v)>0$. Therefore $\alpha(v)> 0$ for some $\alpha \in \Phi'_{0}$. This is a contradiction to $\Cone(\Phi'_{0})=\ker(\Phi'_{0})$. Once this is true, for any $\alpha\in \Phi'_{0}$ for $a>0$ small enough, $-a \alpha$ is still in the cone and can be written as non-negative linear combinations of elements from $\Phi'_{0}$. By the definition of $\Phi_0$, we have $\alpha\in \Phi_0$.
\end{proof}

Built on this lemma, we can show that 
\begin{lem}\label{lemonPolytope}
Let $V$ be a finite-dimensional $\R$-vector space equipped with an Euclidean metric and $\Phi\subset V^*$ be a finite collection of linear functionals on $V$. 
For each $\bm{a}\in \Map(\Phi,\R)$, we define
$\Omega(\Phi,\bm{a}):=\{
v\in V,\,\alpha(v)\geq \bm{a}(\alpha) \,\forall\alpha\in\Phi
\}$.
Assume that we are given a decomposition $\Phi=\Phi_0\sqcup\Phi_1\sqcup\Phi_{\infty}$ and a sequence of
$\{\bm{a}_n\} \subset \Map(\Phi,\R_{\geq 0})$ satisfying:
\begin{itemize}
    \item [(1)] there exists $\bm{a}_0 \in \Map(\Phi_0\sqcup\Phi_1,\R)$ such that 
                 $\bm{a}_n\vert_{\Phi_0\sqcup\Phi_1} = \bm{a}_0$ for all $n$;
    \item [(2)] for all $\alpha \in \Phi_{\infty}$, $\bm{a}_n(\alpha)$ diverges to $+\infty$;
    \item [(3)] $\Phi_1$ and $\Phi_0$ are compatible with Definition \ref{defiofphi} where $\Phi_{bdd}:=\Phi_1\sqcup\Phi_0$.
\end{itemize}
Then we can find $\{\omega_n\}$, a diverging sequence of positive numbers, such that 
\begin{equation*}
    \lim_{n\to \infty} 
    \frac{\Vol(
    \Omega(\Phi_{\infty}\sqcup \Phi_1, -\bm{a}_n+\omega_n)
    \cap
     \Omega(\Phi_{0}, -\bm{a}_n)
    )}
    {\Vol(
     \Omega(\Phi, -\bm{a}_n)
    )} =1.
\end{equation*}
In fact, let $U$ be the orthogonal complement of $W=W(\Phi)$ in $V$ and denote by $\pi^W_U$ the orthogonal projection onto $U$.
Then there exists $\{\omega'_n\}$ and $\{\omega_n\}$ two diverging sequence of positive numbers, such that if we define
\begin{equation*}
    \Omega^{split}_n:=\pi^W_{U}(\Omega(\Phi_0,-\bm{a}_0) ) \oplus
    \big(
    W \cap 
    \Omega(\Phi_{\infty}\sqcup \Phi_1,-\bm{a}_n+\omega'_n)
    \big),
\end{equation*} then
\begin{equation*}
    \lim_{n\to \infty} 
    \frac{\Vol(
    \Omega^{split}_n
    )}
    {\Vol(
     \Omega(\Phi, -\bm{a}_n)
    )} =1
\end{equation*}
and 
$\Omega^{split}_n$ is contained in 
$\Omega(\Phi_{\infty}\sqcup \Phi_1, -\bm{a}_n+\omega_n) \cap \Omega(\Phi_{0}, -\bm{a}_n)$ for n large enough.
\end{lem}

When $\Phi_{0}=\emptyset$, this has been treated in \cite{ShaZhe18}. 
For simplicity write $\text{Top}_n(\omega_n)$ for 
$\Omega(\Phi_{\infty}\sqcup \Phi_1, -\bm{a}_n+\omega_n) \cap \Omega(\Phi_{0}, -\bm{a}_n)$ 
so $\Omega(\Phi, -\bm{a}_n)= \Top_n(0)$. 
We also write $U_0$ for $\pi^W_U(\Omega(\Phi_0,-\bm{a}_0))$, then $\Omega^{split}_n= U_0 + (W\cap \Top_n(\omega_n'))$.

The reader may find it helpful to keep the following example in mind.
Take $V=\R^3$ with standard basis $\{\bme_1,\bme_2,\bme_3\}$ and write its dual basis as $\{\bmf_1,\bmf_2,\bmf_3\}$. 
Let $\Phi:=\{\bmf_3, -\bmf_3, -\bmf_1-\bmf_2, \bmf_1-\bmf_2, \bmf_2\}$,  
$\bma_{n}(\bmf_3)=0$, $\bma_{n}(-\bmf_3)=n$ and
$\bma_{n}(-\bmf_1-\bmf_2)=\bma_{n}(\bmf_1-\bmf_2)=\bma_{n}(\bmf_2)=-1$.
Then one can check that $\Phi_{\infty}=\{-\bmf_3\}$, $\Phi_1=\{\bmf_3\}$ and $\Phi_0=\{ -\bmf_1-\bmf_2, \bmf_1-\bmf_2, \bmf_2\}$. Also $W=\R\bme_3$. And $\Omega(\Phi, -\bm{a}_n)$ is a cylinder based on a fixed triangle with the ceiling keeping growing and the floor remaining fixed. The projection $\pi_W$ in this case can be regarded as crushing the cylinder into its base triangle.
    
\begin{proof}
It suffices to prove the second asymptotic and the claim on the last line. First we note that
\subsubsection{Claim}\label{claSameBase}  For every sequence $\{\omega_n\}$ of real numbers such that $\bma_n(\alpha)-\omega_n$ diverges to $+\infty$ for all $\alpha\in\Phi_{\infty}$, for $n$ large enough,
$\pi_W(\text{Top}_n(\omega_n))= \pi_W(\Omega(\Phi_0,-\bm{a}_0))$ and both are bounded. Consequently, the same thing is true replacing $\pi_W$ by $\pi^W_U$. 
\begin{proof}[Proof of Claim \ref{claSameBase}]
The non-trivial direction is to show that $\pi_W(\text{Top}_n(\omega_n))$ contains $\pi_W(\Omega(\Phi_0,-\bm{a}_0))$.
First we claim that $\pi_W(\Omega(\Phi_0,-\bma_0))$ is bounded. 
Indeed $\Phi_0$ descends to a set of functionals $\overline{\Phi_0}$ on $V/W$. 
If not bounded then we can find a ray $\R_{\geq 0}\cdot x $ in $\pi_W(\Omega(\Phi_0,-\bma_0))$, that is, $\overline{\alpha}(rx)\geq -\bma_0$ for all $r \geq 0$ and $\alpha \in \overline{\Phi_0}$.
Hence $\overline{\alpha}(x)\geq 0$ for all $\alpha \in \overline{\Phi_0}$. As there are strictly positive numbers $a_{\alpha}$ such that $\sum a_{\alpha}\alpha=0$. 
We conclude that $\overline{\alpha}(x) = 0$, hence $x=0$ by definition of $W$.

Now we take a compact set $B \subset \Omega(\Phi_0,-\bma_0)$ such that $\pi_W(B)=\pi_W(\Omega(\Phi_0,-\bma_0))$. As $\Phi$ is a finite set, we can find $b>0$ such that $\alpha\vert_{B}>-b$ for all $\alpha\in\Phi$. 
We also take $v_0\in W$ as in the proof of Lemma \ref{lemmaStrucCone} such that $\alpha(v_0)>0$ for all $\alpha\in \Phi_1$.
Then we can find $r_0>0$ such that $\alpha(v+r_0v_0)> -\bma_0$ for  all $\alpha \in \Phi_1$ and $v\in B$. 
Now there exists a possibly different $b'>0$ such that $\alpha\vert_{B+r_0v_0}>-b'$ for all $\alpha\in\Phi_{\infty}$. 
Take $n$ such that $\bma_n(\alpha)-\omega_n> b'$ and $v \in B $, we show that $v+W\cap \text{Top}_n(\omega_n)$ is non-empty, which is sufficient to conclude the proof.
Indeed for $\alpha \in \Phi_{\infty}$, $\alpha(v+r_0v_0)>-b'>-\bma_n(\alpha)+\omega_n$. And inequalities for $\Phi_1$ and $\Phi_0$ have already been verified. So we are done.
\end{proof}
Hence for any choice of $\omega_n$ with $\bma_n(\alpha)-\omega_n$ diverging to $+\infty$ and for $n$ large enough,
\begin{equation*}
\begin{aligned}
        \Vol(\Top_n(\omega_n))
    = \int_{u\in U_0}
    \Vol((u+W) \cap \Top_n(\omega_n))
    \text{d}u.
\end{aligned}
\end{equation*}

As $U_0$ is bounded we can find a number $\omega_0>0$ such that for every choice of $\{\omega_n\}$ and each $u\in U_0$,
\begin{equation*}
\begin{aligned}
        (u+W) \cap \Top_n(\omega_n)
        &\supset
    u+ (W\cap \Top_n(\omega_n+\omega_0)),
    \\
    (u+W) \cap \Top_n(0)
    &\subset
    u+ (W\cap \Top_n(-\omega_0) ).
\end{aligned}
\end{equation*}
In particular the first containment implies that $\Omega^{split}_n= U_0+(W\cap \Top_n(\omega_n'))$ is contained in $\Top_n(\omega_n)$ if we define $\omega_n':=\omega_n+\omega_0$ and $n$ is large enough(depending on the choice of $\omega_n$).

And the second containment implies that 
\begin{equation*}
    \frac{\Vol (U_0+(W\cap \Top_n(\omega_n')))
    }
    { \Vol (\Top_n(0))
    }
    \geq 
    \frac{\Vol (U_0+(W\cap \Top_n(\omega_n')))
    }
    { \Vol(U_0+(W\cap \Top_n(-\omega_0)))
    }.
\end{equation*}

As $\Cone(\Phi_{bdd})$ restricted to $W$ is non-empty and open, arguing as in \cite[Lemma 6.2, 9.4]{ShaZhe18}, we know that there exists a divergent sequence of positive numbers $\{\omega'_n\}$ such that 
\begin{equation*}
\lim_{n\to \infty}
\frac{\Vol (u_0+(W\cap \Top_n(\omega_n')))
    }
    { \Vol(u_0+(W\cap \Top_n(-\omega_0)))
    } =1
\end{equation*} 
for all $u_0\in U_0$.
Plugging into the integration expression above yields that 
\begin{equation*}
    \lim_{n\to \infty} \frac{\Vol (U_0+(W\cap \Top_n(\omega_n')))
    }
    { \Vol(U_0+(W\cap \Top_n(-\omega_0)))
    }\geq 1.
\end{equation*}
So we are done.
\end{proof} 

\subsection{Step I}\label{sec3PartI}
We fix a superfaithful $\Q$-representation $\rho$ of $\bmL$. 
We assume that the sequence $\{\gamma_n\}$ is $\Phi$-clean for some $\Phi\in \Phi_{\rho,\bmL}$.
Take $W=W(\Phi_{bdd})$ defined to be $\ker\Phi_0$ and $U$ to be the orthogonal complement of $W$. 
Moreover we require $\rho$ to contain a direct summand of the form
\begin{equation*}
    \bmL \to \bmS_{\bmL} \to \SL_N
\end{equation*}
with the second arrow being faithful. In particular, 
\begin{itemize}
    \item there exists a basis $\{\widetilde{\alpha}_1,...,\widetilde{\alpha}_n\} $ of $X^*(\bmS_{\bmL})\otimes \Q$ consisting of $\Q$-characters appearing in $\Phi_{\rho}$ and positive numbers $m_1,...,m_n>0$ such that $\sum m_i \widetilde{\alpha}_i =0$.
\end{itemize}
Without loss of generality we assume that $m_i$'s are actually positive integers.
The dependence on $\rho$ will often be dropped as it will be fixed throughout this subsection.

Apply Lemma \ref{lemonPolytope} to $\Phi$, $V=\Lie(S_{\bmH})$ and 
$\displaystyle \bma_n(\diffalpha)= -\ln\ep +\ln\inf_{0\neq v\in \bmV_{\alpha_n}(\Z)} ||v||$ with $\alpha_n:=(\bmc_{\gamma_n})_{*}\alpha$. Also $\Phi_{*}=\Phi_{*}(\{\gamma_n\})$ for $*=0,1,\infty$.
By passing to a subsequence we assume that $\bma_n(\diffalpha)$ either diverges to $+\infty$ or remains constantly equal to some $\bma_0(\diffalpha)$. 

Keep the notation $U$, $W$ and $U_0$ as in Lemma \ref{lemonPolytope}. As $U_0$ depends on $\ep$, we shall write it as $U_0(\ep)$. Note that as $\ep$ decreases to $0$, $U_0(\ep)$ forms an increasing family of polytopes whose union covers $U$. 

\begin{lem}
There exists a sequence of real numbers $\omega_n \to +\infty$ such that if we define
\begin{equation*}
\begin{aligned}
         \Omega^{spl}_{n,\ep}:= U_0(\ep) \oplus (W\cap \Omega(\bmc_{\gamma_n},\ep+\omega_n,\Phi_1\sqcup\Phi_{\infty}))
\end{aligned}
\end{equation*}
then $\Omega^{spl}_{n,\ep}$ is contained in $\Omega(\bmc_{\gamma_n},\ep,\Phi)$ for $n$ large enough and
\begin{equation*}
    \lim_{n\to \infty}
    \frac{\Vol(\Omega^{spl}_{n,\ep})}
    {\Vol(
    \Omega(\bmc_{\gamma_n},\ep,\Phi)
    )
    } =1.
\end{equation*}
\end{lem}
We also define 
\begin{equation*}
   \Omega^{vert}_{n,\ep} := W\cap \Omega(\bmc_{\gamma_n},\ep+\omega_n,\Phi_1\sqcup\Phi_{\infty})
\end{equation*} and
\begin{equation*}
    \calP^{spl}_{n,\ep}:= \{h\in H \mid 
    \pi_{{}^{\circ}\bmH}(h) \in \exp(\Omega^{spl}_{n,\ep})
    \}.
\end{equation*}

The homomorphism $\bmc_{\gamma_n}$ induces $\bmp_n : \bmS_{\bmH}\to \bmS_{\bmL}$, which does not depend on $n$ after passing to a subsequence. We shall call this map $\bmp :=\bmp_n$.

\begin{lem}
$\bmp_n=\bmp_m$ if $\gamma_n$ and $\gamma_m$ are in the same Zariski connected component of $X(\bmH,\bmL)$.
\end{lem}

\begin{proof}
Note that $X(\bmH,\bmL)$ is an affine variety. Apply \cite[Proposition 3.2.8]{Spr98}.
\end{proof}

\begin{lem}\label{pFactorthrough}
By abuse of notation we also write $\bmp : S_{\bmH}\to S_{\bmL}$ as a morphism of Lie groups. 
Then ${\bmp}$ factors through $\overline{\bmp}:  S_{\bmH}/\exp(W) \to S_{\bmL}$. $\overline{\bmp}$ is an isomorphism if $(\{\gamma_n\},\bmL)$ is minimal for $\bmH$.
\end{lem}

\begin{proof}
 Recall that we have fixed a $\Q$-basis $\{\widetilde{\alpha}_1,...,\widetilde{\alpha}_n\} $ of characters of $\bmS_{\bmL}$ appearing in $\Phi_{\rho}$ and positive numbers $m_1,...,m_n>0$ such that $\sum m_i \widetilde{\alpha}_i =0$. 
 It suffices to show that $\widetilde{\alpha}_i (\bmp(\exp{w}))=1$ for all $w\in W$ and for all $i$. 
 For this it suffices to show that $\alpha_i:= \pi_{{}^{\circ}\bmH}^*\circ \bmp^* (\widetilde{\alpha}_i)$ lies in $\Phi_0(\{\gamma_n\})$. 
 By the very definition of $\Phi_0$, we only need to prove that they are in $\Phi_{bdd}(\{\gamma_n\})$ 
 which also follows from the definition as 
 $(\bmc_{\gamma_n})_{*}\alpha_i=\pi_{{}^{\circ}\bmL}^* \widetilde{\alpha}_i$ 
 is independent of $n$.
 
 The surjectivity follows by assumption and we now show that $\overline{\bmp}$ is injective. So take $s_0\in S_{\bmH}$ and a lift $h_0$ of $s_0$ in $H$. 
 Assume that $\bmc_{\gamma_n}(h_0)$ is contained in $^{\circ}\bmL$ for all $n$, we want to show that $s_0$ is contained in $\exp(W)$. 
 So take $a$ in $\Phi_0$(actually $\Phi_{bdd}$ suffices). Passing to a subsequence we assume that there exists non-zero $v_a\in\bmV(\Z)$ such that $\bmc_{\gamma_n}(h)v_a=a(h)v_a$ for all  $n$ and $h\in \bmH$. 
 By assumption the line spanned by $v_a$ is also preserved by $\bmL$, so $lv_{a}=b(l)v_a$ for all $l\in\bmL$ and for some character $b$ of $\bmL$. That $b(\bmc_{\gamma_n}h_0)= \pm 1 $ and $b(L)$ is contained in $\R^+$ imply that $b(\bmc_{\gamma_n}h_0)=1$. So $a(h_0)=1$ and $\text{da}(\log(s_0)) =0$. And we are done. 
 
\end{proof}

In summary we now have the following commutative diagram:
\begin{equation*}
    \begin{tikzcd}
        \calP^{spl}_{n,\ep} \arrow[hook]{r} \arrow[d]
        & H \arrow[d] \arrow[r,"\bmc_{\gamma_n}"]
        & L \arrow[d]
         \\
        \calP^{spl}_{n,\ep}/ \Gamma_{H}  \arrow[hook]{r}  \arrow[d]
        &  H / \Gamma_{H}\arrow[d,""] \arrow[r,"\bmc_n"]
        & L/\Gamma_L \arrow[d,  "" ]
        \\
        \exp(\Omega^{spl}_{n,\ep}) \arrow[hook]{r}  \arrow[d]
        & S_{\bmH} \arrow[d, "\pi_W"] \arrow[r,"\bmp"]
        & S_{\bmL}
        \\
        \exp(U_0(\ep)) \arrow[hook]{r}  
        & (S_{\bmH})/\exp(W)  \arrow[ru,dotted,swap,"\overline{\bmp}"]
    \end{tikzcd}
\end{equation*}
and the most important point here is that the bottom row does not depend on $n$.

\subsubsection*{Normalization of Haar measures}
We normalize the Haar measure $\mu_H$ on $H/\Gamma_H$ and $\mu_L$ on $L/\Gamma_L$ such that $\mu_H$(resp. $\mu_L$) can be written as the fibre integration of probability homogeneous measures of ${}^{\circ}H$(resp.  ${}^{\circ}L$) over the base $(S_{\bmH},\mu_{S_{\bmH}})$(resp. $(S_{\bmL},\mu_{S_{\bmL}})$) such that $\overline{\bmp}_*\mu_{S_{\bmH}/\exp{W}}= \mu_{S_{\bmL}}$ where $\mu_{S_{\bmH}/\exp{W}}$ is induced from the quotient metric in the Lie algebra. Also, we assume that the $\Vol$ on the Lie algebra of $S_{\bmH}$ is the same as $\mu_{S_{\bmH}}$ under the exponential map.

Now we come to the main proposition of this subsection. For each $\delta>0$, let $\calO_{\delta}=\{h\in H\,\vert\, d(h,e)\leq \delta \}$.

\begin{prop}\label{partIofsec3}
Given a standard triple $(\bmG,\bmH,\Gamma)$, a sequence $\{\gamma_n\}\subset \Gamma$ and a connected $\Q$-subgroup $\bmL \leq \bmG$. Assume that $(\{\gamma_n\},\bmL)$ is minimal for $\bmH$.
Consider the map $\bmc_n: H/\Gamma_H \to L/\Gamma_L$ induced from $\bmc_{\gamma_n}$.
Then $\lim_n a_n(\gamma_n)_*\mu_H = \mu_L$ in $L/\Gamma_L$
with $a_n= 1/\Vol(\Omega^{vert}_{n,\ep})$ for all $\ep>0$ small enough.
\end{prop}

\begin{proof}
Let $\eta>0$ be an arbitrary small number.
We can find $W_0 \subset W$ depending on $\eta$ such that for each $n$ there exists 
$I_N =\{t^N_i\}_{i=1,...,l_n}\subset W$ such that 
\begin{equation*}
    \Omega^{spl}_{n,\ep} \approx_{\eta\Vol(\Omega)} \bigsqcup_{t \in I_N} U_0(\ep)\oplus (t+W_0)
\end{equation*}
where ``$\approx_{\eta\Vol(\Omega)}$'' means that the measure of the symmetric difference between both sides are smaller than $\kappa_{\eta}\Vol(\Omega^{spl}_{n,\ep})$ for some constant $\kappa_{\eta}$ decreasing to $0$ as $\eta$ does.
Hence we can also find $\calO\subset \calP^{spl}_{n,\ep}$($\calO$ is independent of $n$, it depends on $\eta$ and $\ep$) such that $\calO$ maps onto $U_0(\ep)$ under the natural projection and 
there exists $J_N=\{h^N_{j}\}_{j=1,...,l'_n} \subset H$ contained in the preimage of $\exp{W}$ such that 
\begin{equation*}
    \calP^{spl}_{n,\ep}/\Gamma_H \approx_{\eta\Vol(\calP)} \bigsqcup_{h \in J_N} h\calO\Gamma_H/\Gamma_H.
\end{equation*}
Let us take arbitrary $k_n\in J_N$ for each $n$ and let $h_n:=h^N_{k_n}$. 
Let $\nu$ be a limit of $\{(\bmc_n h_n )_{*}\widehat{\mu}_{\calO}\}$. 
So we can find an infinite subsequence $\{n_k\}$ such that 
\begin{equation*}
    \nu= \lim_{k \to \infty} (\bmc_{n_k} h_{n_k} )_{*}\widehat{\mu}_{\calO}.
\end{equation*}
Then by Proposition \ref{nondiverg}, $\nu$ is a probability measure on $L/\Gamma_L$ and
$ \bmc_{\gamma_{n_k}}( h_{n_k} \calO_{\delta})$ is non-divergent in $L/\Gamma_L$ for each $\delta>0$ small enough, i.e, intersects a compact set in $L/\Gamma_L$ non-trivially for all $k$. Hence there exist
\begin{itemize}
    \item $\{\gamma_k'\} \subset \Gamma_L$, bounded $\{\delta_k\}\subset L$ and $\{o_k\}\subset \calO_{\delta}$ such that
    \item $\gamma_{n_k} h_{n_k} o_k \gamma_{n_k}^{-1} = \delta_k \gamma_k'$.
\end{itemize}
Let $\lambda_k:= \gamma_k' \gamma_{n_k}$ and $\bmc_k'$ be the map from $H/\Gamma_H \to L/\Gamma_L$ induced from $\bmc_{\lambda_k}$. Then
\begin{equation*}
    \gamma_{n_k} h_{n_k} o \gamma_{n_k}^{-1}  = \delta_k \lambda_k o_k^{-1}o \lambda_k^{-1} \gamma_k', 
    \,\, \forall o\in \calO
\end{equation*}
So we have 
\begin{equation*}
    (\bmc_{n_k}\circ h_{n_k})_* \widehat{\mu}_{\calO} \approx_{\delta} (\delta_{k}\circ \bmc'_{k})_* \widehat{\mu}_{\calO}
\end{equation*}
where $\approx_{\delta}$ means the total mass of the symmetric difference of both sides is smaller than $\kappa_\delta$ for some $\kappa_{\delta}$ converging to $0$ as $\delta$ does.
For simplicity write $\pi$ for the natural projection from $L/\Gamma_L$ to $S_{\bmL}$.
By Corollary \ref{coroEMS} we assume that by passing to a subsequence $k_i$, 
\begin{equation*}
    \lim_{i\to \infty}  (\bmc'_{k_i})_* \widehat{\mu}_{\calO} =: \nu' 
    = \int_{S_{\bmL}} \mu_{\pi^{-1}(s)} {\pi}_*\nu'(s)
\end{equation*}
where each $\mu_{\pi^{-1}(s)}$ is the unique probability ${}^{\circ}L$-invariant measure supported on $\pi^{-1}(s)$ and $\nu'$ is a probability measure. We may also assume that $\delta_{k_i}$ converges to some $\delta_{\infty}$ and so 
\begin{equation*}
    \lim_{i\to \infty}  (\delta_{k_i}\bmc'_{k_i})_* \widehat{\mu}_{\calO} =: \nu'' 
    = \int_{S_{\bmL}} \mu_{\pi^{-1}(s)} \pi_*\nu''(s).
\end{equation*}

Now
\begin{equation*}
    \begin{aligned}
         \pi_* \nu'' &= 
         \lim_{i\to \infty} (\pi\circ \delta_{k_i} \circ \bmc_{k_i}')_* \widehat{\mu}_{\calO}
         \approx_{\delta} \lim_{i\to \infty} 
         (\pi\circ \bmc_{n_{k_i}} \circ h_{n_{k_i}})_{*} \widehat{\mu}_{\calO}\\
         &= \lim_{i\to \infty} (\bmp \circ \pi \circ h_{n_{k_i}})_{*} \widehat{\mu}_{\calO}
         = \lim_{i\to \infty} (\overline{\bmp}\circ \pi_W \circ \pi \circ h_{n_{k_i}})_{*} \widehat{\mu}_{\calO}\\
         &= \overline{\bmp}_* \widehat{\mu}_{U_0(\ep)} = \widehat{\mu}_{S_{\bmL}}\vert_{\overline{\bmp}(U_0(\ep))},
    \end{aligned}
\end{equation*}
where we have employed the commutative diagram above.
As the output is independent of the subsequence chosen and the constant $\delta>0$, 
by letting $\delta$ converge to $0$, we actually have 
\begin{equation*}
     \lim_{n \to \infty} (\bmc_{n} h_{n} )_{*}\widehat{\mu}_{\calO} = 
     \int_{S_{\bmL}} \mu_{\pi^{-1}(s)}  \widehat{\mu}_{S_{\bmL}}\vert_{\overline{p}(U_0(\ep))}.
\end{equation*}
By taking average,
 \begin{equation*}
     \lim_{n\to \infty} 
     \frac{(\bmc_n)_* \mu_{H}\vert_{\calP^{spl}_{n,\ep}}
     }
     {\mu_H ({\calP^{spl}_{n,\ep}}/\Gamma_H ) } \approx_{\eta}
      \int_{S_{\bmL}} \mu_{\pi^{-1}(s)}  \widehat{\mu}_{S_{\bmL}}\vert_{\overline{p}(U_0(\ep))}.
 \end{equation*}
  Also, by our normalization of Haar measures, 
  \begin{equation*}
      \mu_{H}({\calP^{spl}_{n,\ep}}/\Gamma_H)= \Vol(\Omega^{spl}_{n,\ep})
  =\Vol(U_0(\ep)) \cdot \Vol(\Omega^{vert}_{n,\ep}).
  \end{equation*}
  Hence, by letting $\eta \to 0$, we have
   \begin{equation*}
     \lim_{n\to \infty} 
     \frac{(\bmc_n)_* \mu_{H}\vert_{\calP^{spl}_{n,\ep}}
     }
     {\Vol(\Omega^{vert}_{n,\ep} ) }
     =
      \int_{S_{\bmL}} \mu_{\pi^{-1}(s)}  {\mu}_{S_{\bmL}}\vert_{\overline{p}(U_0(\ep))}.
 \end{equation*}
 Note that for any two different $\ep,\ep'>0$, the asymptotic of $\Vol(\Omega^{vert}_{n,\ep} ) $ remains the same. So by fixing such an $\ep_0>0$ and let $\ep$ go to zero we get
  \begin{equation*}
       \lim_{n\to \infty} 
     \frac{1}{\Vol(\Omega^{vert}_{n,\ep_0})}(\bmc_n)_* \mu_{H}= \mu_{L}.
  \end{equation*}
\end{proof}

\subsubsection{Example}
The perhaps most basic example is to take 
$\bmG=\bmL=\left[
  \begin{array}{cccc}
    a & *    \\
    0 & 1/a
  \end{array}
  \right]$, 
  $\bmH=\left[
  \begin{array}{cccc}
    a & 0 \\
    0 & 1/a
  \end{array}
  \right]$, $\Gamma= \left[
  \begin{array}{cccc}
    1 & \Z    \\
    0 & 1
  \end{array}
  \right]$ and $\gamma_n= 
  \left[
  \begin{array}{cccc}
    1 & n    \\
    0 & 1
  \end{array}
  \right]$.  It is not hard to check that $(\{\gamma_n\},\bmL)$ is indeed minimal for $\bmH$. So our theorem asserts that $(\gamma_n)_*[\mu_H]\to [\mu_G]$. Let us see why this is true intuitively.
  We first take a model of $G \to G/\Gamma$:
  \begin{equation*}
        \begin{tikzcd}
            \R\times \R  
             \arrow{d}{\pi}\arrow{r}{\Phi}
             & G
            \arrow{d}{\pi_{\Gamma}}
           \\
           \R \times \R/\Z
           \arrow{r}{\phi}
           & G/\Gamma \\
        \end{tikzcd} 
  \end{equation*}
where $\pi$ is the natural quotient map and $\Phi(x,y):= \left[
  \begin{array}{cccc}
    e^x & 0   \\
    0 & e^{-x}
  \end{array}
  \right]\left[
  \begin{array}{cccc}
    1 & y    \\
    0 & 1
  \end{array}
  \right]$. The little $\phi$ is then induced from $\Phi$. Under this isomorphism, the set $\gamma_n H$ becomes $\{(x,y)\mid y=e^{-2x}n\}$ carried with the measure $\diffx$. 
  And we may view $\pi_{\Gamma}(\gamma_n H)$ as a cord wrapping about an infinite cylinder. The larger the $n$ is, the denser the wrapping becomes. 
  Measure theoretically our theorem in this case is equivalent to the following elementary statement:
  for each compactly supported function $f:\R \to \C$ and each non-zero integer $m$,
  \begin{equation*}
      \int f(x)\exp(2\pi i m n e^{-2x}) \diffx \to 0 \quad \text{as } n \to \infty.
  \end{equation*}

\subsection{Step II}\label{sec3PartII}

A drawback of Proposition \ref{partIofsec3} is that the convergence happens inside $L/\Gamma_L$. In order for the convergence to happen on $G/\Gamma$, it is sufficient that $L\Gamma/\Gamma$ is closed in $G/\Gamma$.
This is true if $\bmL$ is observable in $\bmG$. The converse is also true, as is proved in \cite{Weiss98}. We shall not make use of the latter fact but rather derive it as a corollary.

In this section we fix a superfaithful $\Q$-representation $\rho: \bmG \to \SL_N $ which contains all exterior products of the Adjoint representation of $\bmG$. We write $\Phi$ for $\Phi_{\rho}$.

\begin{lem}\label{obslemma}
Given a standard triple $(\bmG,\bmH,\Gamma)$ and $\rho$ as above. 
Let $\{\lambda_n\}\subset \Gamma$ be a sequence that is $\Phi$-clean. 
Let $\bmL \leq \bmG$ be a connected $\Q$-subgroup. Assume that $(\{\lambda_{n_k}\},\bmL)$ is minimal for $\bmH$ for some infinite subsequence $n_k$. Let $\bmL'$ be the observable hull of $\bmL$ in $\bmG$. 
If $\Phi_1(\{\lambda_n\}) = \emptyset$, then
\begin{enumerate}
    \item A normal $\Q$-subgroup of $\bmL$ is also normal in $\bmL'$;
    \item $\bmL=\bmL'$, i.e., $\bmL$ is observable.
\end{enumerate}
\end{lem}

\begin{proof}
Take $\bmN$ to be a normal $\Q$-subgroup of $\bmL$.
Let $v_{\bm{N}}$ be a non-zero vector in $\bmV(\Z)$ that represents $\Lie(\bm{N})$. Then $g\in \bmG$ normalize $\bm{N}$ iff $g[v_{\bm{N}}]=[v_{\bm{N}}]$ where $[v_{\bm{N}}]$ denotes the line spanned by $v_{\bm{N}}$.
We know by assumption that $\lambda_{n_k}\bmH \lambda_{n_k}^{-1}[v_{\bm{N}}]=[v_{\bm{N}}]$. 
By passing to an infinite subsequence we may assume that for some $\alpha \in \Phi$,
\begin{equation*}
    \lambda_{n_k}\bmH \lambda_{n_k}^{-1} v_{\bm{N}} =\alpha(h) v_{\bm{N}}, \,\, \forall k
\end{equation*}
which implies that $\alpha$ is contained in $\Phi_{bdd}(\{\lambda_n\}) = \Phi_{0}(\{\lambda_n\})$.

So there exists $I=\{\alpha_0=\alpha,\alpha_1,...,\alpha_n\}\subset \Phi_0(\{\lambda_n\})$ 
and $\{a_0,...,a_n\}\subset \Z_{>0}$ such that $\sum a_i \alpha_i =0$. Passing to a further subsequence we may assume that for each $i$, there exists a non-zero $v_i$ that is simultaneously an $\alpha_i$-weight vector of 
$\lambda_{n_k}\bmH \lambda_{n_k}^{-1}$ for all $k$. Let 
$w:= v_{\bm{N}}^{\otimes a_0} \bigotimes \otimes_{i=1}^N v_i^{\otimes a_i}$, then
\begin{equation*}
    \lambda_{n_k} h \lambda_{n_k}^{-1} w = \prod_{i=0}^N \alpha_i(h)^{a_i}(h) w =w, \,\,\forall k
    \implies \bmL w = w \implies \bmL' w= w
\end{equation*}
By Lemma \ref{tensortrick}, $\bmL'[v_{\bm{N}}]=[v_{\bm{N}}]$ so we have proved the first claim.

Now apply this to $\bm{R}_{\bmL}$, the radical of $\bmL$. So $\bm{R}_{\bmL}$ is normal in $\bmL'$. 
Note that $\bmL/\bm{R}_{\bmL} \leq \bmL'/\bm{R}_{\bmL}$ is observable as $\bmL/\bm{R}_{\bmL}$ is semisimple. 
This implies that $\bmL\leq \bmL'$ is also observable so we are done.
\end{proof}

Now we can state and prove the main proposition of this subsection.

\begin{prop}\label{partIIofsection3}
Given an observable standard triple $(\bmG,\bmH,\Gamma)$ and $\rho$ as above. 
Let $\{\lambda_n\}\subset \Gamma$ be a sequence that is $\Phi$-clean.
Let $\bmL$ be a connected $\Q$-subgroup of $\bmG$. 
Assume that $(\{\gamma_{n}\},\bmL)$ is minimal for $\bmH$ and let $\bmF$ be the observable hull of $\bmL$ in $\bmG$. Then 
$\lim_{n\to\infty} a_n(\gamma_n)_* \mu_H = \mu_{F}$ with $a_n$ the same as in Proposition \ref{partIofsec3}.
\end{prop}

\begin{proof}
Take $\ep>0$ small enough and $\{h_n\}\subset \calP^{spl}_{n,\ep}$, then by Proposition \ref{nondiverg}, there exists a bounded sequence $\{o_n\}\subset H$, a bounded sequence
$\{\delta_n\}\subset F$ and $\{\lambda_n \} \subset \Gamma_F$ such that 
\begin{equation*}
    \gamma_n h_n o_n \gamma_n^{-1} = \delta_n \gamma'_n.
\end{equation*}
Define $\lambda_n=\gamma'_n\gamma_n$ and note that $\{\lambda_n\}$ is $\Phi$-clean. 
So we have $\Phi= \Phi_{\infty}(\{\lambda_n\})\sqcup\Phi_{1}(\{\lambda_n\})\sqcup\Phi_{0}(\{\lambda_n\})$.

\subsubsection{Claim}\label{claPhi_1EmptyA}
$\Phi_{\infty}(\{\lambda_n\}) \supset 
\Phi_{\infty}(\{\gamma_n\})\sqcup \Phi_{1}(\{\gamma_n\})$.

\begin{proof}
Take $\alpha\in \Phi_{\infty}(\{\gamma_n\})\sqcup \Phi_{1}(\{\gamma_n\})$. By definition
\begin{equation*}
    \inf_{0\neq v \in \bmV_{\alpha}(\Z)} ||\gamma_n h_n  v|| \to +\infty.
\end{equation*}
So
\begin{equation*}
    \inf_{0\neq v \in \bmV_{\alpha}(\Z)} ||\lambda_n  v|| =
    \inf_{0\neq v \in \bmV_{\alpha}(\Z)} ||\delta_n^{-1}\gamma_n h_n o_n v|| 
    \approx \inf_{0\neq v \in \bmV_{\alpha}(\Z)} ||\gamma_n h_n  v||
    \to +\infty
\end{equation*}
where $\approx$ means ``differed by a bounded error as $n$ varies''
and we are done.
\end{proof}

\subsubsection{Claim}\label{claPhi_1EmptyB}
$\Phi_{1}(\{\lambda_n\})=\emptyset$ and $\Phi_{0}(\{\lambda_n\})=\Phi_{0}(\{\gamma_n\})$.

\begin{proof}
Take $\alpha\in \Phi_{0}(\{\gamma_n\})$, then there exists $\{\alpha_0=\alpha,\alpha_1,...,\alpha_n\}\subset \Phi_{0}(\{\gamma_n\})$ and $\{a_0,...,a_n\}\subset \R_{>0}$ such that $\sum a_i \alpha_i=0$. 
On the other hand, there exists $M>0$ such that $ \inf_{0\neq v \in \bmV_{\alpha_i}(\Z)} ||\gamma_n  v|| \leq M$ for all $i$. Therefore,
\begin{equation*}
\begin{aligned}
          &\inf_{0\neq v \in \bmV_{\alpha_i}(\Z)} ||\gamma_n  h_n v|| \geq \ep, \quad \forall i,n 
          \implies \alpha_i(h_n) \geq \frac{\ep}{M}, \quad \forall i,n\\
          \implies &
          \alpha^{a_0}(h_n) = 
          \frac{1}{\prod_{i\neq 0} \alpha_i^{a_i}(h_n)} \leq (\frac{M}{\ep})^{\sum_{i\neq 0} a_i},
          \quad \forall n
          \\
          \implies& \alpha(h_n)\leq (\frac{M}{\ep})^{\sum_{i\neq 1} a_i/a_0}, \quad\forall n.
\end{aligned}
\end{equation*}
So we have 
\begin{equation*}
    \inf_{0\neq v \in \bmV_{\alpha}(\Z)} ||\lambda_n v|| \approx 
    \inf_{0\neq v \in \bmV_{\alpha}(\Z)} ||\gamma_n  v|| \alpha(h_n)
\end{equation*}
is bounded from above, implying that $\Phi_{0}(\{\gamma_n\})$ is contained in  $\Phi_{bdd}(\{\lambda_n\})$.
By definition of $\Phi_0$, this actually implies that
$\Phi_{0}(\{\gamma_n\}) \subset  \Phi_{0}(\{\lambda_n\})$. 
The asserted equalities then come from the fact that 
\begin{equation*}
     \Phi=\Phi_{\infty}(\{\lambda_n\})\sqcup\Phi_{1}(\{\lambda_n\})\sqcup\Phi_{0}(\{\lambda_n\})
     = \Phi_{\infty}(\{\gamma_n\})\sqcup\Phi_{1}(\{\gamma_n\})\sqcup\Phi_{0}(\{\gamma_n\}).
\end{equation*}
\end{proof}
Take an arbitrary infinite subsequence $\{n_k\}$ and a $\Q$-subgroup $\bmL' \leq \bmG$ such that $(\{\lambda_{n_k}\},\bmL')$ is $\bmH$-minimal. We may assume $\delta_{n_k}$ converges to some $\delta_{\infty}$ in $F$.
Then $\bmL'$ is observable by Lemma \ref{obslemma} and
\begin{equation*}
    \lim_{k\to \infty} [(\lambda_{n_k})_* \mu_H] =[\mu_{L'}] \implies
    \lim_{k\to \infty} [(\gamma_{n_k})_* \mu_H] =(\delta_{\infty})_*[\mu_{L'}]
\end{equation*} in $G/\Gamma$.
Now we claim that $\bmL'=\bmF$, which would conclude the proof. It is clear that $\bmL'$ is contained in $\bmF$. It is sufficient to show that it is also epimorphic. Indeed take a $\Q$-representation of $\bmF$ and a non-zero $\Q$-vector $v$ fixed by $\bmL'$, i.e., $\lambda_n \bmH \lambda_n^{-1}v=v$ for all $n$. 
Note $\lambda_n=\delta_n^{-1}\gamma_nh_no_n$, 
so $\lambda_n \bmH \lambda_n^{-1}= \delta_n^{-1}\gamma_n \bmH \gamma_n^{-1}\delta_n$. Hence $\delta_n v$ is fixed by $\gamma_n \bmH \gamma_n^{-1}$.
On the other hand 
\begin{equation*}
    \delta_n v = \gamma_n h_no_n \gamma_n^{-1} \gamma_n'^{-1} v = \gamma_n'^{-1} \lambda_n  h_no_n \lambda_n^{-1} v
    = \gamma_n'^{-1} v
\end{equation*} is both discrete and bounded. Hence by passing to a subsequence we may assume that there exists $w$, another $\Q$-vector, such that $w=\delta_nv=\gamma_n'^{-1} v$. 

Now $w$ is fixed by  $\gamma_n \bmH \gamma_n^{-1}$ for all $n$ and so it is fixed by $\bmF$. Therefore $v$ is also fixed by $\bmF$. So we are done.
\end{proof}

\subsection{Complements}

\begin{defi}\label{defiLambConv}
Given an observable standard triple $(\bmG,\bmH,\Gamma)$ and a connected observable $\Q$-subgroup $\bmL$. Let $\Lambda$ be a subgroup of $\Gamma$. $\bmH$ is said to \textbf{$\Lambda$-converge to $\bmL$} iff there exists a sequence $\{\lambda_n\}$ in $\Lambda$ such that 
$(\{\lambda_n\},\bmL)$ is potentially minimal for $\bmH$.
\end{defi}
By Proposition \ref{partIIofsection3}, if $\bmH$ $\Lambda$-converges to $\bmL$ then $(\lambda_n)_*[\mu_H]$ converges to $[\mu_L]$ for some sequence $\{\lambda_n\}$ of $\Lambda$. And the converse is also true by ignoring finitely many $n$'s.
One may ask when $\bmH$ could $\Lambda$-converge to $\bmL$.  We shall make some observations here but will not be able to answer the general question even when $\Lambda=\Gamma$. 
 
 \begin{lem}\label{lemTransiLambConv}
  Keep the notations as in the above definition. Given three connected observable $\Q$-subgroups $\bmA$, $\bmB$ and $\bmC$ of $\bmG$. Then $\bmA$ $\Lambda$-converges to $\bmB$ and $\bmB$ $\Lambda$-converges to $\bmC$ implies that $\bmA$ $\Lambda$-converges to $\bmC$. As a consequence, $\bmA\leq_{\Lambda} \bmB$ iff $\bmA$ $\Lambda$-converges to $\bmB$ defines a partial order on the set of all connected observable $\Q$-subgroups of $\bmG$.
 \end{lem}
 \begin{proof}
 By assumption we can find a sequence $\{a_n\}$(resp. $\{b_n\}$) in $\Lambda$ such that 
 $(\{a_n\},\bmB)$(resp. $(\{b_n\},\bmC)$) is potentially minimal for $\bmA$(resp. $\bmB$). Let $\scrD$ be the collection of connected observable $\Q$-subgroups of $\bmC$. It is a countable set and we fix a enumeration $\scrD=\{\bmD_1,\bmD_2,...\}$. 
 
 Then for each positive integer $i$ there exists $N_i\in \Z^+$ such that for all $n\geq N_i$, $b_n\bmB b_n^{-1}$ is not contained in $\bmD_j$ for all $j\leq i$. 
 For each fixed $n$, $(\{b_na_m\}_m, b_n\bmB b_n^{-1})$ is minimal for $\bmA$. 
 So we can find $M_{n,i}\in \Z^+$ such that for all $n\geq N_i$ and $m\geq M_{n,i}$, $b_na_m \bmA a_m^{-1}b_n^{-1}$ is not contained in $\bmD_j$ for all $j\leq i$(here we use the fact that $\bmD_j \cap b_n\bmB b_n^{-1}$ is an observable subgroup of strictly smaller dimension than $b_n\bmB b_n^{-1}$). 
 Hence if we define $c_i = b_{n_i}a_{m_i}$ in $\Lambda$ for some $n_i\geq N_i$ and $m_i\geq M_{n,i}$, then $(\{c_n\},\bmC)$ is potentially minimal for $\bmA$.
 \end{proof}
 
 \begin{lem}\label{lemMaxLamdaConn}
 Notations are the same as above. Assume that $\bmH$ is an observable subgroup of $\bmG$ and is maximal(among proper subgroups) with respect to $\leq_{\Lambda}$. Then $\bmH$ is virtually normalized by $\Lambda$ in the sense that there is a finite-index subgroup $\Lambda_0$ of $\Lambda$ that normalizes $\bmH$.
 \end{lem}
 
\begin{proof}
If the conclusion is false, then we can find a sequence $\{\lambda_n\}$ in $\Lambda$ such that $\lambda_n\bmH \lambda_n^{-1}\neq \lambda_m\bmH \lambda_m^{-1} $ for any $n \neq m$. Passing to an infinite subsequence we may assume there exists a connected $\Q$-subgroup $\bmL$ of $\bmG$ such that $(\{\lambda_n\},\bmL)$ is minimal for $\bmH$. 
But by assumption $\dim \bmL = \dim \bmH$, so for all $n$, $\bmL=\lambda_n\bmH \lambda_n^{-1}$ and this is a contradiction.
\end{proof}

\begin{coro}\label{coroLambdaConvSimple}
 Notations are the same as above. Assume that $\bmG$ is $\Q$-simple and $\Q$-isotropic. Then the only maximal elements with respect to $\leq_{\Gamma}$ are the trivial group and $\bmG$. Consequently for each non-trivial connected observable $\Q$-subgroup $\bmL$ of $\bmG$, there exists a sequence $\{\gamma_n\}$ in $\Gamma$ such that 
 $\lim_n (\gamma_n)_* [\mu_L] = [\mu_G]$.
\end{coro}

\begin{proof}
Take such a maximal element $\bmL$. Then $\bmL$ is normalized by some finite-index subgroup $\Gamma_0$ of $\Gamma$. But Borel's density theorem implies that $\Gamma_0$ is Zariski-dense in $\bmG$. Hence $\bmL$ is normalized by $\bmG$, which must be either the trivial group or $\bmG$.
\end{proof}
 
 We also deduce a result of Weiss \cite[Corollary 5]{Weiss98}.
 
 \begin{coro}\label{coroWeiss}
  Given a standard triple $(\bmG,\bmH,\Gamma)$ and assume $\bmH$ to be epimorphic in $\bmG$, then $\pi_{\Gamma}(H)$ is dense in $G/\Gamma$. Therefore $\pi_{\Gamma}(H)$ is closed iff $\bmH$ is observable.
 \end{coro}
 
 \begin{proof}
 Take a Levi decomposition of $\bmH= \bmL \cdot \bmU$ where $\bmL$ is reductive and hence observable in $\bmG$. Let $\Lambda= U_{\Gamma}$ and take $\bmF$ to be a connected observable $\Q$-subgroup such that $\bmL$ $\Lambda$-converges to $\bmF$ and $\bmF$ is maximal with respect to this property. 
 By Lemma \ref{lemMaxLamdaConn} and that any finite-index subgroup of $\Lambda$ is Zariski dense in $\bmU$(see \cite[Theorem 2.1]{Ragh72}), we see that $\bmF$ is normalized by $\bmU$.
 
 Now we use item $(6)$ in Definition \ref{defiObs} to check that $\bmF\cdot \bmU$ is observable in $\bmG$. 
 Only here we use $\bmV$ to stand for $\bmV(\overline{\Q})$ instead of $\bmV(\C)$.
 Take a character $\alpha$ of $\bmF\cdot \bmU$ that extends to a representation of $\bmG$. 
 In other words, there is a representation $(\rho_1,\bmV_1)$ of $\bmG$ and non-zero $v\in\bmV $ such that $xv=\alpha(x)v$ for all $x\in \bmF\cdot \bmU$. As $\bmF$ is observable, by $(6)$ in Definition \ref{defiObs}, there exists another representation $(\rho',\bmV')$ such that the subspace $\bmV'_{-\alpha|_{\bmF}}$ is non-zero. As $\bmU$ normalize $\bmF$, $\bmU$ preserves the subspace $\bmV'_{-\alpha|_{\bmF}}$. But $\bmU$ is a unipotent group so there exists a non-zero $v'\in \bmV'_{-\alpha|_{\bmF}}$ that is fixed by $\bmU$. Hence $-\alpha$, which is just the dual of $\alpha$, is also contained in a representation of $\bmG$.
 
 As $\bmF\cdot \bmU$ contains $\bmH$ by definition, $\bmF\cdot \bmU$ is equal to $\bmG$ due to maximality. Hence we also have $G=F\cdot U$.
 
 Now we take $X$ to be the closure of $\pi_{\Gamma}(H)$ in $G/\Gamma$. It is invariant by $U$. Hence it contains translates of $L\Gamma/\Gamma$ by $\Lambda$ and by Proposition \ref{partIIofsection3} also contains $F\Gamma/\Gamma$. So $X\supset U F\Gamma/\Gamma=G/\Gamma$ and the proof completes.
 \end{proof}

\section{Translates of reductive subgroups}\label{secRed}

In this section we prove Theorem \ref{thmTransReduc}. 
Let $(\bmG,\bmH,\Gamma)$ be a standard triple and assume that both $\bmG$ and $\bmH$ are reductive groups. 

\subsection{Non-divergence}
We want to put an assumption on $\bmH$ that would guarantee the non-divergence when translated by an arbitrary sequence in $G$.
If $\bma$ is a $\Q$-cocharacter in $\bmG$ that centralize $\bmH$ yet not contained in $\bmH$, then the full orbit $\pi_{\Gamma}(\bma_t H)$ diverges to infinite set-theoretically, that is, for any compact subset $K$ of $G/\Gamma$ and for $t$ large enough, $\pi_{\Gamma}(\bma_t H)\cap K=\emptyset$. If one wants to avoid this scenario, then it is necessary to put the following group-theoretically condition:
\begin{itemize}
     \item $\bmZ_{\bmG}\bmH / (\bmZ_{\bmG}\bmH \cap \bmH)$ is $\Q$-anisotropic.
\end{itemize}

\begin{prop}\label{proNondivReduc}
Under this condition, there exists a compact set $K\subset G/\Gamma$ such that $\pi_{\Gamma}(gH)\cap K \neq \emptyset$ for all $g\in G$.
\end{prop}

In the case when $\bmZ_{\bmG}\bmH$ is $\Q$-anisotropic, this is proved in \cite[Theorem 1.1]{EskMozSha97}. 
And in the present case it can be deduced from \cite{RicSha18} and \cite{KleMar98}.
We shall give an alternative short proof based on Proposition \ref{nondiverg}(which also relies on \cite{KleMar98}) and some input from geometric invariant theory. In view of Proposition \ref{nondiverg} it suffices to prove Proposition \ref{ProUniLowBou} below.

First we record a result from \cite[Corollary 4.5]{Kemp78}.
\begin{lem}\label{lemFixByHisGClose}
Take a $\Q$-representation $\rho: \bmG \to \GL(\bmV)$ and a vector $v_0\neq 0$ in $\bmV(\Q)$ fixed by $\bmH$.
If $\bmZ_{\bmG}\bmH / (\bmZ_{\bmG}\bmH \cap \bmH)$ is $\Q$-anisotropic, then $\bmG\cdot v_0$ is closed.
\end{lem}

Take a $\Q$-representation $\rho: \bmG \to \GL(\bmV)$ and assume $\bmH$ to be reductive. We can find a Cartan involution of $\GL(V)$ that preserves the image of $G$ and $H$(see \cite{Mos552}). Then we take an Euclidean metric on $V$ that is invariant under the maximal compact subgroup associated with this Cartan involution. Note that under this assumption, if $W$ is an $H$-invariant subspace, then $W^{\perp}$ is also $H$-invariant.

\begin{lem}\label{LemMiniAtCen}
With the assumption in the last paragraph, we take a vector $v_0\neq 0$ in $\bmV(\Q)$ fixed by $\bmH$.
If $Z_{\bmG}(\bmH)\cdot v_0$ is closed and $v_1$ is a vector with minimum length in $Z_{\bmG}(\bmH) \cdot v_0$ , then $v_1$ is also a vector with minimum length in $G \cdot v_1$.
\end{lem}

The proof is a modification of the proof in \cite[Theorem 3.27]{Wal17}.

\begin{proof}
We apply Ness' theorem over $\R$
\cite[Theorem 3.28]{Wal17}
which states that $v\in V$ achieves the minimum length of a closed orbit of $H$ (resp. $Z_{\bmG}\bmH$, resp. $G$) iff $(Xv,v)=0$ for all $X$ in $\Lie H$ (resp. $\Lie (Z_{\bmG}\bmH)$, resp. $\Lie (G)$).

So we know that $(Xv_1,v_1)=0$ for all $X\in \Lie(Z_{\bmG}\bmH)$ and we only need to show that this holds for all $X\in \Lie(G)$. 

Consider the following diagram where we denote $\Lie(G)$
by $\frakg$ and $\Lie(Z_{\bmG}\bmH)$ by 
$\frakz_{\frakg}\frakh$:
\begin{center}
    \begin{tikzcd}
\frakg=\frakz_{\frakg}(\frakh) \bigoplus \frakz_{\frakg}(\frakh)^{\perp} \arrow{r} \arrow {rd}
&V= V^H \bigoplus (V^H)^{\perp} \arrow{d}
\\
&V^H
\end{tikzcd}
\end{center}
where the horizontal arrow is defined by sending $X\mapsto Xv_1$ and
the vertical arrow is the natural projection to the first factor and the remaining diagonal arrow is the composition of the other two.  
Hence the diagram is commutative by definition and all arrows are $\R$-linear.
It is also $H$-equivariant.
The vertical arrow is $H$-equivariant as both subspaces are $H$-invariant.
Take $h \in H$, as $hv_1=v_1$ we have $hXh^{-1}v_1= h(Xv_1)$ and so the horizontal arrow is also $H$-equivariant.

There is no component of trivial $H$-representation in $\frakz_{\frakg}(\frakh)^{\perp}$ by definition so it is sent to $\{0\}$ by the diagonal arrow. So the $V^H$ component of $Xv_1$ for $X\in \frakz_{\frakg}(\frakh)^{\perp}$ is trivial. In particular $(Xv_1,v_1)=0$  for $X\in  \frakz_{\frakg}(\frakh)^{\perp}$  and hence this is true for all $X \in \frakg$. 
\end{proof}

\begin{lem}\label{LemUniLowBouFixed}
Assume that $\bmZ_{\bmG}\bmH / (\bmZ_{\bmG}\bmH \cap \bmH)$ is $\Q$-anisotropic. 
Take a representation $\bmV$ of $\bmG$ over $\Q$ and fix a $\Z$-structure of $\bmV_{\Q}$ and an Euclidean metric on $V$.
Then there exists a constant $c>0$ such that for all $v_{\neq 0}\in \bmV^{\bmH}(\Z)$, we have $||G\cdot v|| \geq c$.
\end{lem}

\begin{proof}
Indeed $||Z_{\bmG}\bmH \cdot v|| > c'$ uniformly over $v\in \bmV^{\bmH}(\Z)-\{0\}$ because $\bmZ_{\bmG}\bmH(\Z)$ is cocompact in $Z_{\bmG}\bmH$ mod $H$. 
As any two norms on $V$ are equivalent, we apply Lemma \ref{LemMiniAtCen} to conclude the proof.
\end{proof}

\begin{prop}\label{ProUniLowBou}
Assume that $\bmZ_{\bmG}\bmH / (\bmZ_{\bmG}\bmH \cap \bmH)$ is $\Q$-anisotropic. 
Take a representation $(\rho,\bmV)$ of $\bmG$ over $\Q$ and fix a $\Z$-structure of $\bmV_{\Q}$ and an Euclidean metric on $V$.
Then there exists $\ep>0$ such that $\Omega(g,\ep,\rho,\Phi_{\rho})$ is non-empty for all $g\in G$.
\end{prop}

\begin{proof}

Consider all possible subsets $I=\{\alpha_1,...,\alpha_k\}$ of $\Phi_{\rho}$ such that $\sum_i m_i\alpha_i=0$ has a solution $\{m_i\}$ in positive numbers. 
For each $I$ we fix a set of positive integers $\{m_i\}$ such that $\sum_i m_i\alpha_i=0$. 
As there are at most finitely many such $I$'s, the representation $\bmW:= \bigoplus_{I} \bigotimes \bmV^{\otimes m_i}$ is finite-dimensional.
Apply Lemma \ref{LemUniLowBouFixed} to $\bmW$ we have a constant $c>0$ which is a lower bound for $||gw||$ for all $g\in G$ and $w\in \bmW^{\bmH}(\Z)$.

Note that for each $t\in \Lie(S_{\bmH})$, $\Omega(g,\ep,\rho,\Phi_{\rho})$ is non-empty iff $\Omega(g\exp{t},\ep,\rho,\Phi_{\rho})$ is non-empty. We want to find a $t$ such that $\Omega(g\exp{t},\ep,\rho,\Phi_{\rho})$ contains $0$.
Consider the function 
\begin{equation*}
\begin{aligned}
     \phi: \Lie(S_{\bmH}) &\longrightarrow \R \\
     t \mapsto &\sup_{\alpha\in\Phi_{\rho}}(\ln{\ep}-\ln \inf_{0\neq v\in \bmV_{\alpha}(\Z)}||g\exp{(t)}v||)\\
     =&\sup_{\alpha\in\Phi_{\rho}} (
     -\diffalpha(t)-\ln\ep-\ln \inf_{0\neq v\in \bmV_{\alpha}(\Z)}||gv||)
     )
\end{aligned}
\end{equation*}
So  $\Omega(g,\ep,\rho,\Phi_{\rho})$ contains $0$ iff $\inf_{t\in\Lie(S_{\bmH})} \phi(t)$  is nonpositive.
If $\inf \phi(t)$ is equal to $-\infty$ then we are done. Otherwise $\inf \phi$ can be achieved by some $t_0$. Consider the set $\Phi$ of $\alpha\in\Phi_{\rho}$ that achieves the supreme in the definition of $\phi(t_0)$. 
Define $\Phi_0$ and $\Phi_1$ as before. We claim that $\Phi_0$ is non-empty. 
Otherwise there exists $t$ such that $\diffalpha(t)>0$ for all $\alpha\in \Phi$ and perturbing by such an element would destroy the infimum. 
Hence there are $\{\alpha_1,...,\alpha_k\}$ in $\Phi$ and positive integers $\{m_1,...,m_k\}$ such that $\sum m_i \alpha_i=0$.  Moreover we take the same $\{m_i\}$'s as in the beginning of the proof. 
Thus $\otimes_i v_{i}^{\otimes m_i}$ is contained in $\bmW^{\bmH}(\Z)$ and
\begin{equation*}
    \prod_{i=1,...,k}||g\exp(t)v_i||^{m_i}=||g\otimes_{i=1,...,k}v_{i}^{\otimes m_i}||\geq c
\end{equation*}
for all $0\neq v_i\in \bmV_{\alpha_i}(\Z)$.
This implies that there exists $i_0$ and $c'>0$ such that 
$ \inf_{0\neq v\in \bmV_{\alpha}(\Z)}||g\exp(t)v||\geq c'$. By taking $\ep $ such that $\ln \ep - \ln c' <0$, or equivalently $\ep < c'$, then we are done.
\end{proof}

\subsection{Equidistribution}

In this section we enhance Proposition \ref{partIIofsection3} in the current case.

\begin{prop}\label{ProEquiRed}
Let $(\bmG,\bmH,\Gamma)$ be a standard triple. We assume in addition that $\bmG$, $\bmH$ are both reductive and that $\bmZ_{\bmG}\bmH / (\bmZ_{\bmG}\bmH \cap \bmH)$ is $\Q$-anisotropic. Given an arbitrary sequence $\{g_n\}$ in $G$, after passing to a subsequence, there exists 
a bounded sequence $\{\delta_n\}$ in $G$, a sequence $\{\gamma_n\}$ in $\Gamma$ and a reductive $\Q$-subgroup $\bmL$ of $\bmG$ such that 
$g_n \mu_H= \delta_n \gamma_n \mu_H$, $\gamma_n \bmH \gamma_n^{-1}\subset \bmL$ and $(\gamma_n)_*[ \mu_H] \to [\mu_L]$. Moreover if $\{g_n\}$ is unbounded when projecting to $G/Z_{\bmG}\bmS$ for all nontrivial $\Q$-split subtori $\bmS$ in $\bmZ(H)$, then $\bmL$ is not contained in any proper $\Q$-parabolic subgroup and $\mu_L$ is finite.
\end{prop}

In view of Proposition \ref{partIIofsection3} and \ref{proNondivReduc}, we have such an observable subgroup $\bmL$.
It only remains to show the other claims about $\bmL$. However, it is more convenient to establish the Theorem \ref{corObConRedisRed} after we prove the proposition above(though there are special cases of this corollary where one has a more direct proof, see \cite[Lemma 3.10]{Grosshans97} when $\bmH$ is a maximal torus).

\begin{defi}
For a $\Q$-cocharacter $\bma:\bmG_m \to \bmG$ of a reductive $\Q$-group $\bmG$, we define $\bmP_{\bma}$ to be the $\Q$-parabolic subgroup 
$\{ x \in \bmG \,\vert\,\lim_{t\to 0} \bma_t x \bma_t^{-1} \text{exists}\}$.
\end{defi}

Then the unipotent radical of $\bmP_{\bma}$ is the subgroup
$\{ x \in \bmG, \,\lim_{t\to 0} \bma_t x \bma_t^{-1}= \bme\}$.

\begin{proof}[Proof of Corollary \ref{corObConRedisRed}]
Assume that the conclusion is false and take $\bmF$ to be such a group, which we may assume to be connected. Write $\bmU$ to be the non-trivial unipotent radical of $\bmF$.
By Proposition \ref{ProEquiRed} above if  $\bmH \leq_{\Gamma} \bmF$ then we are done. So assume that this is also false and without loss of generality we assume that $\bmH$ is a subgroup of $\bmF$ that is maximal with respect to $\leq_{\Gamma}$, then it is automatically maximal with respect to any subgroup of $\Gamma$.
In particular this is true for $\Lambda:=\Gamma\cap U$, a lattice in $U$. 
By Lemma \ref{lemMaxLamdaConn}, $\bmH$ is normalized by a finite-index subgroup $\Lambda_0$ of $\Lambda$, which is Zariski dense in $\bmU$. Hence $\bmH$ and consequently its center are normalized by $\bmU$. Let $\bmS$ be the maximal $\Q$-split torus in the center of $\bmH$. Then it follows that $\bmS$ is also normalized by $\bmU$. Hence $\bmS$ is centralized by $\bmU$ as $\bmU$ is connected(see \cite[Corollary 3.29]{Spr98}). 

On the other hand, there exists a proper parabolic $\Q$-subgroup $\bmP$ of $\bmG$ that contains $\bmU$ inside its unipotent radical and also contains $\bmN_{\bmG}\bmU$. As $\bmH$ is normalized $\bmU$, $\bmP$ contains $\bmH$. So by Lemma \ref{lemmaParabolicContainH} below, $\bmP$ is equal to $\bmP_{\bma}$ for some cocharacter $\bma$ of $\bmS$. But this is a contradiction as $\bmU$ is contained in the unipotent radical of $\bmP$ so is impossible to centralize the image of $\bma$.
\end{proof}

\begin{lem}\label{lemmaParabolicContainH}
Let $\bmG$ be a reductive $\Q$-group and $\bmH$ be a reductive $\Q$-subgroup of $\bmG$ such that $\bmZ_{\bmG}\bmH / (\bmZ_{\bmG}\bmH \cap \bmH)$ is $\Q$-anisotropic.
Then any $\Q$-parabolic subgroup $\bmP$ containing $\bmH$ is equal to $\bmP_{\bma}$ for some $\Q$-cocharacter $\bma$ of $\bmS$ where $\bmS$ is the maximal $\Q$-split torus in the center of $\bmH$.
\end{lem}

\begin{proof}
Let $\bmL$ be a Levi subgroup of $\bmP$ that contains $\bmH$ and $\bmS'$ be the maximal $\Q$-split torus of the center of $\bmL$. Then $\bmP$ is equal to $\bmP_{\bma}$ for some $\Q$-cocharacter $\bma$ of $\bmS'$.
But $\bmS'$ centralize $\bmH$, therefore is contained in $\bmS$ by the assumption that $\bmZ_{\bmG}\bmH / (\bmZ_{\bmG}\bmH \cap \bmH)$ is $\Q$-anisotropic.
\end{proof}

Let us now turn to the proof of Proposition \ref{ProEquiRed}. We need \cite[Theorem 4.2]{Kemp78}.

\begin{thm}\label{Kempf}
Let $\bmG$ be a reductive $\Q$-group and $(\rho,\bmV)$ be a $\Q$-representation of $\bmG$. 
For each $v_{\neq 0} \in \bmV(\Q)$ such that $\overline{\bmG \cdot v} \ni \{0\}$, there is a unique $\Q$-parabolic subgroup $\bmP_v$ of $\bmG$ such that for each $\Q$-cocharacter $\bma$ of $\bmG$ that is ``optimal'', we have $\bmP_v= \bmP_{\bma}$.
Moreover,  if $\bmH$ preserves the line spanned by $v$, then $\bmH$ is contained in $\bmP_v$.
\end{thm}

We refer the reader to  \cite{Kemp78} for the precise meaning of ``optimal''. Here we only note that if $v$ is a weight vector with respect to a $\Q$-split subtorus $\bmS$ with non-zero weight $\alpha$ and $\bmG$ is semisimple, then a $\Q$-cocharacter is optimal within the class of $\Q$-cocharacters of $\bmS$ if and only if it is contained in $\Q^+\alpha^{\vee}$ where $\alpha \to \alpha^{\vee}$ denotes the identification of $X_*(\bmS)\otimes \Q$ with $X^*(\bmS)\otimes \Q$ provided by the restriction of Killing form. If it happens that this cocharacter is also optimal in the class of $\Q$-cocharacters of $\bmG$, then we also write $\bmP_{\alpha^{\vee}}$ for $\bmP_v$ in this case. And if we decompose the Lie algebra $\frakg$ of $\bmG$ with respect to the Adjoint action of $\bmS$ as $\oplus_{\beta\in X^*(\bmS)} \frakg_{\beta}$, then the Lie algebra of $\bmP_{\alpha^{\vee}}$ is  
$\displaystyle\oplus_{\beta\in X^*(\bmS), (\beta,\alpha)\geq 0} \frakg_{\beta}.$

\begin{proof}[Proof of Proposition \ref{ProEquiRed}]
We let $\widetilde{\bmS}_{\bmH}$ be the unique lift of ${\bmS}_{\bmH}$ in $\bmH$. This is the maximal $\Q$-split torus in the center of $\bmH$.
By passing to a subsequence, we assume that there is a subtorus $\bmS_0$ of $\widetilde{\bmS}_{\bmH}$ such that 
\begin{itemize}
    \item for all subtori $\bmS$ of $\widetilde{\bmS}_{\bmH}$ that properly contains $\bmS_0$,
     $\{g_n\}$ is unbounded when projecting to $G/Z_{\bmG}\bmS$ and
    \item the sequence $\{g_n\}$ is contained in $Z_{\bmG}\bmS_0$.
\end{itemize}
Hence we may replace the ambient group $\bmG$ by $Z_{\bmG}\bmS_0$. As the center plays no role in the dynamics, we may further replace it by its derived semisimple subgroup $[\bmG,\bmG]$. Under these assumptions, one may check that the sequence $\{g_n\}$ is unbounded when projecting to $G/Z_{\bmG}\bmS$ for all non-trivial subtori $\bmS$ in $\widetilde{\bmS}_{\bmH}$.

By non-divergence, there is a bounded sequence $\{\delta_n\}$ in $G$, a bounded sequence $\{h_n\}$ in $H$ and a sequence $\{\gamma_n\}$ in $\Gamma$ such that $g_nh_n=\delta_n \gamma_n$. Then $(g_n)_*\mu_H= (\delta_n \gamma_n)_* \mu_H $ and $\{\gamma_n\}$ is unbounded when projecting to $G/Z_{\bmG}\bmS$ for all non-trivial subtori $\bmS$ in $\widetilde{\bmS}_{\bmH}$. 
Passing to a subsequence we find a $\Q$-subgroup $\bmF$ such that $(\{\gamma_n\},\bmF)$ is $\bmH$-minimal. 

Recall the notation in previous sections. Take $\rho: \bmG \to \SL(\bmV)$ to be a superfaithful $\Q$-representation that  contains all exterior powers of the Adjoint representation of $\bmG$. Passing to a subsequence we assume that $\gamma_n$ is $\Phi_{\rho}$-clean. Hence $\Phi_0$, $\Phi_1$ and $\Phi_{\infty}$ are defined. By the proof of Proposition \ref{partIIofsection3} we may assume that $\Phi_1$ is empty and, by Lemma \ref{obslemma}, $\bmF$ is observable. It suffices to show that $\bmF$ is not contained in any parabolic $\Q$-subgroup of $\bmG$. As this will firstly imply that $\bmF$ is reductive and \cite[Lemma 5.1]{EskMozSha97} further implies that $\bmF$ has no $\Q$-characters.

Let us first show that $\Phi_0$, and hence $\Phi_{bdd}$, is contained in $\{0\}$. If not, 
then there are non-zero characters $\{\alpha_i\}_{i=1,...,l}$ of $\widetilde{\bmS}_{\bmH}$, positive numbers $\{a_i\}_{i=1,...,l}$ and non-zero vectors $v_n(\alpha_i)\in \bmV_{\alpha_i}(\Z)$ such that $\sum_{i=1}^{l} a_i \alpha_i =0 $ and $\{\gamma_n v_n(\alpha_i)\}_n$ bounded. 
As $\{\gamma_n v_n(\alpha_i)\}_n$ is discrete, we
assume that $\gamma_n v_n(\alpha_i)$ is constantly equal to some non-zero vector $v_i\in \bmV(\Z)$ by passing to a further subsequence. As $\alpha_i$ is non-zero, the $\bmG$-orbit through $v_i$ contains $\{0\}$ in its closure.
According to Theorem \ref{Kempf}, there exists a canonical $\Q$-parabolic subgroup $\bmP_i$ that contains the stabilizer of the line $\Q v_i$. In particular $\bmP_i$ contains $\gamma_n \bmH \gamma_n^{-1}$ for all $n$.

By Lemma \ref{lemmaParabolicContainH} and Theorem \ref{Kempf} above, there exists a cocharacter $\bma^I_N$ of $\widetilde{\bmS}_{\bmH}$ such that  $t \mapsto \gamma_n \bma^I_N(t)\gamma_n^{-1}$ is optimal for $v_i$ and 
$\bmP_i= \gamma_n\bmP_{\bma^I_N}\gamma_n^{-1}$. 
By the remarks made after Theorem \ref{Kempf}, we have $\bmP_{\bma^I_N}=\bmP_{\alpha_i^{\vee}}$, independent of $n$.
Hence $ \gamma_n\bmP_{\alpha_i^{\vee}}\gamma_n^{-1}=\gamma_1\bmP_{\alpha_i^{\vee}}\gamma_1^{-1}$ and so $\gamma_1^{-1}\gamma_n$ normalize $\bmP_{\alpha_i^{\vee}}$. But the normalizer of a parabolic subgroup is equal to itself(see \cite[Corollary 6.4.10]{Spr98}), so $\gamma_1^{-1}\gamma_n$ is contained in $\bmL$, defined to be the intersection  of $\bmP_{\alpha_i^{\vee}}$'s as $i$ ranges from $1$ to $l$.

Referring to the remarks made after Theorem \ref{Kempf} again, the Lie algebra of $\bmL$ consists of $\frakg_{\beta}$ with $(\beta,\alpha_i)\geq 0$ for all $i$. But $\sum_{i=1}^{l} a_i \alpha_i =0 $ for some positive numbers $\{a_i\}_{i=1,...,l}$, we are forced to have $(\beta,\alpha_i) = 0$  for all $i$.
Let $\bmS_{\beta}$ be any $\Q$-split subtorus of $\widetilde{\bmS}_{\bmH}$ which is the image of some non-trivial cocharacter $\bma$ that is in the $\Q$-span of $\beta^{\vee}$. Then the Adjoint action of  $\bmS_{\beta}$ restricted to the Lie algebra of $\bmL$ is trivial. In other words, $\bmL^{\circ}$ is contained in the centralizer of  $\bmS_{\beta}$ . This is a contradiction as $\gamma_1^{-1}\gamma_n$ would then be bounded in $G/Z_{\bmG}\bmS_{\beta}$.

So we have proved $\Phi_{bdd}$ is at most $\{0\}$.
Now suppose that $\bmP$ is a proper parabolic $\Q$-subgroup of $\bmG$, we want to show that $\bmF$ is not contained in $\bmP$. If this were not true, we have that $\gamma_n \bmH \gamma_n^{-1}$ normalizes the unipotent radical $\bmU$ of $\bmP$ for all $n$. Let $v_{\bmU}$ be an integral vector in $\wedge^{\dim \bmU}\frakg$ that represents $\bmU$, then  $\gamma_n \bmH \gamma_n^{-1}$  stabilize the line spanned by  $v_{\bmU}$ for all $n$. Then we see that the character of $\bmH$ associated to $\gamma_n^{-1} v_{\bmU}$ is in $\Phi_{bdd}$, which has been shown to be contained in $\{0\}$. But this is a contradiction to Lemma \ref{lemmaParabolicContainH} above.
\end{proof}

\section{Examples and Applications}\label{SecExamAppl}

\subsection{Examples}
In this subsection we prove Proposition \ref{propMaxTori} and \ref{propToriSym}.

\begin{proof}[Proof of Proposition \ref{propMaxTori}]
If the proposition were not true, by Theorem \ref{thmTransReduc}, we can find a sequence $\{\gamma_n\}\subset{\Gamma}$ satisfying the property in the proposition, such that $\bmL$ contains $\gamma_n \bmH \gamma_n^{-1}$ for all $n$. By taking a conjugate, we may assume that $\bmL$ contains $\bmH$. By \cite{Gil10}(compare \cite{BorDeS49}), there exists a diagonalizable subgroup $\bmZ$ of $\bmH$ that contains the center properly such that $\bmL$ is the connected component of $\bmZ_{\bmG}\bmZ$. 

Therefore, $\gamma_n\bmH\gamma_n^{-1}\subset \bmZ_{\bmG}(\bmZ)$, or equivalently, 
$\bmH\subset \bmZ_{\bmG}(\gamma_n^{-1}\bmZ\gamma_n)$ for all $n$. Hence $\gamma_n^{-1}\bmZ\gamma_n$ centralize the maximal torus $\bmH$, and hence is contained in $\bmH$. Therefore $\gamma_n \in \bmX(\bmZ,\bmH)$ for all $n$. Apply \cite[Proposition 3.2.8]{Spr98} and pass to a subsequence, we may assume that 
\begin{equation*}
    \gamma_1 z \gamma_1^{-1}=\gamma_n z \gamma_n^{-1},\,\forall n,\,\forall z\in\bmZ.
\end{equation*}
Hence $\gamma_n \gamma_1^{-1} \in \bmZ_{\bmG}(\bmZ)$, or equivalently, $\gamma_n \in \bmZ_{\bmG}(\gamma_1^{-1}\bmZ\gamma_1)$ for all $n$.  Note that $\gamma_1^{-1}\bmZ\gamma_1$ is also a subgroup of $\bmH$ properly containing the center. And this is a contradiction.
\end{proof}

Let $J$ be the anti-symmetric matrix 
\begin{equation*}
    \left[
 \begin{array}{c|c}
0  &  J_N
 \\
\hline
-J_N  &  0
\end{array}
\right],\quad
J_N =
\left[ \begin{array}{ccc}
  & & 1\\
 & \iddots & \\
 1& & 
\end{array}
\right].
\end{equation*}
Let $\bmG=\Sp_{2N}$ be the subgroup of $\SL_{2N}$ which preserves the symplectic form represented by $J$. 

\begin{proof}[Proof of Proposition \ref{propToriSym}]
It suffices to show that for each diagonalizable subgroup $\bmZ$ of $\bmH$ there exists $\bmT_{\bmZ} \leq \bmH$ a $\Q$-subtorus such that $\bmZ_{\bmG}(\bmZ)=\bmZ_{\bmG}(\bmT_{\bmZ})$. 

First we assume that $\bmH$ is the diagonal torus. 
Writing each element $z\in\bmZ$ as $\diag(z_1,...,z_N,z_N^{-1},...,z_1^{-1})$, we define an equivalence relation on $\{1,...,N\}$ by
\begin{equation*}
    i \sim j \iff z_{i}=z_{j},\,\forall z\in \bmZ.
\end{equation*}
This relation defines a torus $\bmT_{\bmZ}$ in $\bmH$ by 
\begin{equation*}
    \diag(t_1,...,t_N,t_N^{-1},...,t_1^{-1}) \in \bmT_{\bmZ} \iff t_i = t_j,\forall i\sim j.
\end{equation*}
By linear algebra, their centralizer in $\GL_{2N}$ agrees and hence also agrees in $\bmG$. So we are done in this case.

In general, we conjugate $\bmH$ to the diagonal torus(over $\overline{\Q}$) to obtains such a subtorus $\bmT_{\bmZ}$ with the desired property except that $\bmT_{\bmZ}$ may not be defined over $\Q$. But this would imply that $\bmZ(\bmZ_{\bmG}(\bmZ))^{\circ}$ contains $\bmT_{\bmZ}$, which contains $\bmZ$. So we may replace $\bmT_{\bmZ}$ by $\bmZ(\bmZ_{\bmG}(\bmZ))^{\circ}$.
\end{proof}

\subsection{Counting some conjugacy classes in $\Sp_{2N}$ } 

In this subsection for a group $G$, $\mu_G$ will be a Haar measure on $G$ and we instead let $\mu_{G/\Gamma}$ be a Haar measure in the quotient. Let $\bmG$ be the same symplectic group as in last subsection.

Note that $\bmG(\R)$ is automatically connected, so $G=\bmG(\R)$ contains $\bmG(\Q)$.
Let $\Gamma:= \bmG(\Z)$ and $\bmG':=\SL_{2N}$. Let $\bmH$ be the full diagonal torus given by
\begin{equation*}
     \left[
 \begin{array}{c|c}
  \begin{array}{ccc}
  t_1 & &\\
   & \ddots & \\
    &  & t_N
  \end{array}
  &  0
 \\
\hline
0  &  
\begin{array}{ccc}
  t_N^{-1} & &\\
   & \ddots & \\
    &  & t_1^{-1}
  \end{array}
\end{array}
\right].
\end{equation*}
This is a $\Q$-split maximal torus of $\bmG$. And similarly we let $\bmH'$ be the full diagonal torus in $\bmG'$.
Let $\bmU$ be the $\Q$-subgroup of $\bmG$ consisting of elements of the form
\begin{equation*}
     \left[
 \begin{array}{c|c}
  \begin{array}{ccc}
  1 & \cdots & *\\
   & \ddots &  \vdots\\
    &  & 1
  \end{array}
  &   \scalebox{3}{$*$} 
 \\
\hline
0  &  
 \begin{array}{ccc}
  1 & \cdots & *\\
   & \ddots &  \vdots\\
    &  & 1
  \end{array}
\end{array}
\right].
\end{equation*}
 And similarly we let $\bmU'$ be the group of upper triangular unipotent matrices in $\bmG'$. Let $K:= G\cap \SO_{2N}(\R)$ be a maximal compact subgroup of $G$. So we have $(K \times U \times H, \mu_K \otimes \mu_U \otimes \mu_H) \cong (G,\mu_G)$ via $(k,u,h)\mapsto kuh$. 
 
\subsubsection{Normalization of Haar measures}
So far Haar measures are only up to a positive scalar. Now we want to specify this scalar.

First we choose $\mu_K$ to be a probability measure. We take $\mu_H := \frac{2^{(N^2+N)/2-1}}{\prod_{k=1}^N \xi(2k)}\cdot |\wedge_{i=1}^N \frac{\difft_i}{t_i}|$
where $\xi(z)=\pi^{-\frac{z}{2}}\Gamma(\frac{z}{2})\zeta(z)$ is the completed Zeta function.
The reason for this normalization will become clear in a moment.
For the Haar measure on $\bmU$, let us first observe that $\bmU$ consisting of 
\begin{equation*}
 \left[
 \begin{array}{c|c}
   A &  B
 \\
\hline
0  &  
J_N {}^tA^{-1} J_N
\end{array}
\right]
\end{equation*}
with $A$ being arbitrary upper triangular unipotents and $J_NA^{-1}B$ being symmetric. Let $A=(a_{ij})$ and $B=(b_{ij})$. Note that $U$ is in bijection with $\R^{\frac{N(N-1)}{2}+\frac{N(N+1)}{2}}=\R^{N^2}$ via $(a_{ij})_{i<j} \times (b_{ij})_{i+j\geq N+1}$. Then one can show that 
$|\wedge_{i<j}\diffa_{ij}||\wedge_{i+j\geq N+1}\diffb_{ij}|$ is a Haar measure on $U$ and we take $\mu_U$ to be this one.  Now we define $\mu_G$ via the Iwasawa decomposition above, and $\mu_{G/\Gamma}$ to be the quotient measure. We need a result of Siegel \cite[Theorem 11]{Sig43}.

\begin{thm}
Under this normalization, $\mu_{G/\Gamma}$ is a probability measure.
\end{thm}

\begin{proof}
This is not exactly how it is stated in \cite[Theorem 11]{Sig43} as different normalizations of volume forms are used. Let us recall the original statement. Let
\begin{equation*}
    \calH^N:=\{Z=X+iY \,\vert\, {}^tX=X,\, {}^tY=Y,\, Y>0 \}
\end{equation*}
where $X$, $Y$ are both real matrices and $Y>0$ means positive definite.
$\calH^N$ is equipped with a volume form $\omega$, whose absolute value at $i I_N$ is equal to
\begin{equation*}
    2^{\frac{N(N-1)}{2}}|\wedge_{i\leq j}\diffy_{ij}|\wedge|\wedge_{i\geq j}\diffx_{ij}|.
\end{equation*} 
In general, it is induced from the metric $\Tr{(Y^{-1}\diffZ Y^{-1}\diffZ)}$.
We are using a different symplectic group than the ``homogeneous symplectic group $\Omega_0$'' used in \cite{Sig43}, which is defined to preserve the symplectic form 
\begin{equation*}
    \left[
 \begin{array}{c|c}
0 &  I_N
 \\
\hline
-I_N  &  0
\end{array}
\right].
\end{equation*}

Note that
\begin{equation*}
  \left[
 \begin{array}{cc}
I_N &  0
 \\
0  &  J_N
\end{array}
\right]
 \left[
 \begin{array}{cc}
0 &  J_N
 \\
-J_N  &  0
\end{array}
\right]
 \left[
 \begin{array}{cc}
I_N &  0
 \\
0  &  J_N
\end{array}
\right]
=
 \left[
 \begin{array}{cc}
0 &  I_N
 \\
-I_N  &  0
\end{array}
\right].
\end{equation*}
So if we take $M_0$ to be $\diag(I_N,J_N)$ then $A \mapsto M_0^{-1}AM_0$ gives a bijection from $G$ to $\Omega_0$ and from $\Gamma$ to what is called the ``modular group'' in \cite{Sig43}. And an element in $\Omega_0$ acts on $\calH^N$ by
\begin{equation*}
     \left[
 \begin{array}{cc}
A &  B
 \\
C  &  D
\end{array}
\right] \cdot Z:=  (AZ+B)\cdot(CZ+D)^{-1}.
\end{equation*}
With this action and the volume form above, \cite[Theorem 11]{Sig43} states that $\Vol(\Gamma\bs\calH^N)=2\prod_{k=1}^N \xi(2k)$ .
 
 The inverse map  $(u,h)\mapsto (h^{-1},u^{-1})$ from $U\times H$ to $H\times U$ sends the measure $\mu_U\otimes \mu_H$ to $\mu_H \otimes \mu_U$. 
Also, the map $\Phi: (h,u)\mapsto M_0huM_{0}^{-1}\cdot i I_N $, which identifies $H \times U$ with $\calH^N$, pulls back $\frac{1}{2\prod_{k=1}^N \xi(2k)} |\omega|$ to some multiple of $\mu_H\otimes \mu_U$ viewed as the absolute value of a volume form. It remains to show these two (absolute value of) volume forms are the same at $(I_{2N},I_{2N})$.

As before write 
\begin{equation*}
    u =
    \left[
 \begin{array}{cc}
A &  B
 \\
0  &  J_N{}^tA^{-1}J_N
\end{array}
\right]
\end{equation*}
and $h = \diag(t_1,...,t_N,t_N^{-1},...,t_1^{-1})$. Then $\Phi(h,u)$ is equal to
\begin{equation*}
     \left[
 \begin{array}{ccc}
t_1 &   & \\
 &\ddots &\\
0  &  & t_N
\end{array}
\right] B J_N {}^tA \left[
 \begin{array}{ccc}
t_1 &   & \\
 &\ddots &\\
0  &  & t_N
\end{array}
\right] + 
i \left[
 \begin{array}{ccc}
t_1 &   & \\
 &\ddots &\\
0  &  & t_N
\end{array}
\right] A{}^tA \left[
 \begin{array}{ccc}
t_1 &   & \\
 &\ddots &\\
0  &  & t_N
\end{array}
\right].
\end{equation*}
Take its differential and evaluate at $(I_{2N},I_{2N})$ we find 
\begin{equation*}
    (\diffB) J_N + i (\diffA + \diff{}^tA) + 2\diag(\difft_1,...,\difft_N).
\end{equation*}
This means that $\Phi^*_{I_{2N},I_{2N}}$ pulls back $\diffx_{ij}$ to $\diffb_{i, N+1-j}$, $\diffy_{ij}$ with $i<j$ to $\diffa_{ij}$ and $\diffy_{ii}$ to $2\difft_i$. Hence $\Phi^*_{I_{2N},I_{2N}}$ pulls back $|\omega|$ to 
\begin{equation*}
    2^{\frac{N(N-1)}{2}+N}|\wedge_{i+j\geq N+1}\diffb_{ij}|\wedge|\wedge_{i< j}\diffa_{ij}|\wedge|\wedge \difft_i|,
\end{equation*}
and we are done.
\end{proof}

We also identify $(G/H,\mu_{G/H})$ with $(K \times U, \mu_K\otimes \mu_U)$.

\subsubsection{Reduction of the counting problem}

As in the introduction, fix a polynomial $p(t)$ of the form $\prod_{i=1}^N (t^2 -d_i^2)$ with $d_i\in \Z^+$ distinct. 
Consider $\bmX(\R):= \{X\in (\fraksp_{2N})_{\R},\, \det(tI_{2N}-X)=p(t) \}$. 
And let $\bmX(\Z)$ be its intersection with integral matrices. 
Similarly we define $\bmX'(\R)$ and $\bmX'(\Z)$ to be those contained in $\fraksl_{2N}$.
We fix a base point $x_0 \in\bmX(\Z) $ equal to $\diag(d_1,...,d_N,-d_N,...,-d_1)$. Let $||\cdot||$ be the Euclidean norm on $2N$-by-$2N$ matrices and $B_R$ consist of elements in $\bmX(\R)$ with norm less than or equal to $R$.

Note that $g \mapsto g\cdot x_0$ gives an homeomorphism from $G/H$ to $\bmX(\R)$. The non-trivial part is surjectivity. Indeed every element $x$ in $\bmX(\R)$ is regular in $\bmG$, so its centralizer is a maximal $\Q$-torus $\bmT_{x}$ in $\bmG$. On the other hand $x$ is semisimple with distinct $\Q$-eigenvalues, hence is diagonalizable over in $\SL_{2N}(\Q)$. So $\bmT_{x}$ is actually a $\Q$-split torus. Thus there exists $\gamma \in \bmG(\Q)\subset G$ such that $\gamma \bmT_{x} \gamma^{-1}=\bmH$(see \cite[Theorem 15.2.6]{Spr98}). At this point, $\gamma \cdot x $ is already diagonal. It remains to observe that there exists $w\in\bmN_{\bmG}\bmH (\Q)\subset \bmG(\Q)$ such that $w \gamma \cdot x=x_0$.

So we may identify $\bmX(\R)$ with $G/H \cong K \times U$. As $K$ preserves $B_R$ and $\mu_K$ is a probability measure,
we also think of $B_R$ as a subset of $U$ and $\mu_{G/H}(B_R)= \mu_U(B_R)$.

Note that $\Gamma$ preserves $\bmX(\Z)$ and decompose $\bmX(\Z)$ into finitely many orbits by a theorem Borel--Harish-Chandra \cite[Theorem 6.9]{BorHar62}. We let $C_0$ be this number. So to prove Theorem \ref{thmCountSym}, it suffices to count each individual orbit separately. For simplicity we assume that the orbit is $\Gamma \cdot x_0$. Other orbits can be treated similarly(see the last section of \cite{ShaZhe18} or \cite{Zha18} for details).

For each $I \subset \{1,...,2N\}$, we say $I$ is isotropic if the subspace generated by $\{e_i\}_{i\in I}$ is. Let 
\begin{equation*}
    \scrA:=\{I \subset \{1,...,2N\} \,\vert\, I \text{ is isotropic } \}.
\end{equation*}
$\scrA$ allows an explicit description: each $I\in \scrA$ is of the form $J\sqcup 2N+1-J'$ for some disjoint $J\sqcup J'\subset\{1,...,N\}$.
There is a bijective correspondence between $\scrA$ and $\scrP_{\bmH}$ given by taking the stabilizer of the line spanned by $e_{I}:=\wedge_{i\in I} e_i$(again, defined up to sign).
Also there exists $C>0$ such that if $I$ corresponds to $\bmP$, then for all $g\in G$, $\frac{1}{C}||ge_I||<||d_{\bmP}(g)||< C ||ge_I||$. Hence in the definition of $\Omega_{g,\eta}$ as in Equation \ref{EquPolyMaxTori}, we may replace $d_{\bmP}(g)$ by $||ge_I||$ as $I$ varies over $\scrA$.

Now we define two other constants relevant to the counting problem. The reason for this definition will be clear in the proof.

For $I=\{i_1,...,i_k\}$, define $c_I:= \sum_{\lambda=1}^k i_{\lambda}-\lambda$.
Define $C_1>0$ by
\begin{equation*}
    \Vol(
     \{t\in \Lie(H),\,  \diffalpha_I(t) \geq  -c_I 
    ,\,\forall I\in \scrA
    \}
    )
\end{equation*}
where $\Vol$ is required to be compatible with $\mu_H$ under the exponential map.
Define $C_2>0$ by:
\begin{equation*}
\begin{aligned}
    &C_2 \cdot |\prod_{1\leq i<j\leq N}(d_j-d_i)\prod_{1\leq i\leq j\leq N}(d_j+d_i)| \\
    =& 
    \Vol(
     \{ 
    (y_{ij},z_{ij}) \in \R^{N^2}\,\vert\,
    2\sum_{i<j} |y_{ij}|^2 + 2\sum_{i+j> N+1} |z_{ij}|^2 + \sum_{i+j= N+1}|z_{ij}|^2 \leq 1
    \}
    ).
\end{aligned}
\end{equation*}

Let $N_R:= C_1 C_2 R^{N^2} (\ln{R})^{N}$. 
Arguing as in \cite{EskMcM93}, to prove Theorem \ref{thmCountSym} with the constant $C:=C_0 \cdot C_1\cdot C_2$, it suffices to show $\{B_R\}$ is well-rounded(see Lemma \ref{SympleConstantC2}) and the following.
\begin{prop}\label{PropSymCount}
For any $f\in C_{c}(G/\Gamma)$ non-negative, 
\begin{equation*}
    \lim_{R\to+\infty}\frac{1}{N_R} \int_{K \times B_R}\int_H f(kuh) \diffh \diffu \diffk  = \langle f, \mu_{G/\Gamma}  \ra
\end{equation*}
where for simplicity we have written $\diffk $ for $\mu_K$, $\diffu$ for $\mu_U$ and $\diffh$ for $\mu_H$.
\end{prop}

\subsubsection{Different coordinates for $U$ and $B_{R,\ep}$}

The coordinates $a_{ij}$ and $b_{ij}$'s are not easy to work with in terms of $B_R$. So we move to a different set of coordinates and show how the Haar measure can be expressed in this new set of coordinates. Indeed $g \mapsto g \cdot x_0$ gives a bijection from $U$ to
\begin{equation*}
\left\{ X=\left[
\begin{array}{c|c}
     Y  &   Z
     \\
     \hline
      0  &  -J_N {}^{t}YJ_N
\end{array}
\right]
,\,
y_{ii}=d_i,\, y_{ij}=0,\, \forall i>j,\,ZJ_N \text{ is symmetric }
\right\}.
\end{equation*}
In terms of coordinates $x_{ij}$, the coefficients of $X$, this says that $x_{ij}=0$ if $i>j$; $x_{ii}= d_i$ and $x_{N+i,N+i}=-d_{N+1-i}$ if $1\leq i\leq n$; $x_{i,N+j}=x_{N+1-i,N+(N+1-j)}$ for all $1\leq i,j\leq n$.

For instance, when $n=3$, these are
\begin{equation*}
     \left[
 \begin{array}{c|c}
  \begin{array}{ccc}
  d_1 & y_{12} & y_{13}\\
   & d_2 &  y_{23}\\
    &  & d_3
  \end{array}
  &   
    \begin{array}{ccc}
   z_{33} & z_{23} & z_{13}\\
   z_{32 }& z_{22} &  z_{23}\\
    z_{31}& z_{32}  & z_{33}
  \end{array}
 \\
\hline
0  &  
 -\left(\begin{array}{ccc}
  d_3 & y_{23} & y_{13}\\
   & d_2 &  y_{12}\\
    &  & d_1
  \end{array}\right)
\end{array}
\right].
\end{equation*}

Use a computation made in \cite[Lemma 11.2]{ShaZhe18}, we have that
\begin{lem}
 Under the bijection above,
 \begin{equation*}
 |\wedge_{i<j}\diffa_{ij}||\wedge_{i+j\geq N+1}\diffb_{ij}|=
     \frac{
     |\wedge_{i<j}\diffx_{ij}||\wedge_{i+j\geq N+1}\diffz_{ij}|
     }{
     |\prod_{1\leq i<j\leq n}(d_j-d_i)\prod_{1\leq i\leq j\leq n}(d_j+d_i)|
     }.
 \end{equation*}
\end{lem}

Using the coordinates $x_{ij}$ or $(y_{ij},z_{ij})$, $B'_R$ and $B_R$ are
\begin{equation*}
\begin{aligned}
        B'_R:=& \{(x_{ij})_{i<j} \,\vert\,
    \sum_{i<j} |x_{ij}|^2 + 2\sum |d_i|^2 \leq R^2
    \}, \\
    B_R =& 
    \{
    (y_{ij})_{i<j},(z_{ij})_{i\geq j} \,\vert\, \\
    &2\sum_{i<j} |y_{ij}|^2 + 2\sum_{i+j> N+1} |z_{ij}|^2 + \sum_{i+j= N+1}|z_{ij}|^2+ 2\sum |d_i|^2 \leq R^2
    \} 
\end{aligned}
\end{equation*}
and for $\ep>0$ we define 
\begin{equation*}
\begin{aligned}
     B'_{R,\ep}:=&
     \{
     (x_{ij}) \in B'_R \,\vert\,
     |y_{i,i+1}|\geq \ep R,\,\forall 1\leq i\leq 2N
     \}, \\
        B_{R,\ep}:=& \{ 
    (y_{ij},z_{ij}) \in B_R\,\vert\,
    |y_{i,i+1}|\geq \ep R,\,\forall 1\leq i\leq n,\, |z_{n,1}| \geq \ep R
    \}.
\end{aligned}
\end{equation*}
So $B_R$($B_{R,\ep}$ resp.) is just the intersection of $B'_R$($B'_{R,\ep}$ resp.) with $\fraksp_{2N}(\R)$.

Similar to \cite[Lemma 11.6]{ShaZhe18} we have that:
\begin{lem}\label{SympleConstantC2}
For any $\ep>0$, there exists $\ep'>0$ such that 
\begin{equation*}
    \lim_{R\to\infty} \frac{\mu_U(B_{R,\ep'})}{\mu_U(B_{R})} \geq 1- \ep
\end{equation*}
for all $R>1$. And 
\begin{equation*}
    \lim_{R\to\infty} \frac{\mu_U(B_R)}{C_2R^{N^2}} =1.
\end{equation*}
In particular, the family $\{B_R\}$ is well-rounded in the sense of \cite{EskMcM93}.
\end{lem}

\subsubsection{Proof}
Now we come back to the proof of the main proposition.
\begin{lem}\label{lemSymCountMainTerm}
 For any $\ep'>0$, $\delta>0$, $f\in C_c(G/\Gamma)$, there exists $R_0$ such that for all $R>R_0$ and for all $g\in B_{R,\ep}$, $k\in K$, we have
 \begin{equation*}
     \left|
     \frac{1}{C_1(\ln{R})^N} \int f(kuh) \mu_H(h) - \la f,\mu_{G/\Gamma} \ra
     \right|<\delta.
 \end{equation*}
\end{lem}
 
 \begin{lem}\label{lemSymCountErrorTerm}
  For all $f\in C_c(G/\Gamma)$, there exists $C_f>0$ such that for all $R>1$ and all $u\in B_R$, $k\in K$, we have 
  \begin{equation*}
      \left| \frac{1}{C_1(\ln{R})^N} \int f(kuh) \mu_H(h)  \right|
      \leq C_f.
  \end{equation*}
 \end{lem}

Assuming these Lemmas(to be proved latter), let us prove the proposition.

\begin{proof}[Proof of Proposition \ref{PropSymCount}]
For any $\ep>$ and choose $\ep'>0$ according to Lemma 1.3 above. Also fix an arbitrary $\delta>0$. Find $R_0$ by Lemma 1.4 and $C_f$ by Lemma 1.5. We decompose the original integral into two parts:
\begin{equation*}
    I_1 :=\frac{1}{N_R} \int_{K \times B_{R,\ep}}\int_H f(kuh) \mu_H(h)
\end{equation*}
and 
\begin{equation*}
    I_2 :=\frac{1}{N_R} \int_{K \times B_R-B_{R,\ep}}\int_H f(kuh) \mu_H(h).
\end{equation*}
Take $R>R_0$, then 
\begin{equation*}
\begin{aligned}
    I_1 &= \frac{1}{\mu_U(B_R)}\int \la f,\mu_{G/\Gamma} \ra + o(1) \\
    &= \frac{\mu_U(B_{R,\ep'})}{\mu_U(B_R)} (\la f,\mu_{G/\Gamma} \ra + o(1)) \\
    &= (1+o'(1)) (\la f,\mu_{G/\Gamma} \ra + o(1)) 
\end{aligned}
\end{equation*}
where $o(1)$ is some number whose absolute value is less than $\delta$ and $o'(1)$ is some number whose absolute value is less than $\ep$. 

For the other part,
\begin{equation*}
    |I_2| \leq \frac{1}{\mu_U(B_R)}\int_{B_R-B_{R,\ep'}} C_f \leq \ep C_f.
\end{equation*}

Therefore by taking the limit($\limsup$ and $\liminf$) of $I_1+I_2$ and then letting $\ep,\delta$ go to zero we are done.
\end{proof}

\subsubsection{Proof of Lemmas}
To prove Lemma \ref{lemSymCountMainTerm} and \ref{lemSymCountErrorTerm} above, let us recall the computation in \cite[Proposition 11.5, 11.8]{ShaZhe18}.
\begin{prop}\label{PropZhengEstimate}
(1) There exists $M>0$ such that for all $R>1$, for all $g\in B'_R$,
\begin{equation*}
    \ln{||ue_I||} \leq  M \ln{R},\quad \forall I\in\scrA;
\end{equation*}
(2) For all $\ep'>0$, there exists $M_{\ep'}>0$, for all $R>1$ and $u\in B'_{R,\ep'}$, one has
\begin{equation*}
    |
    \ln{||ue_I||-c_I \ln{R}}
    | \leq M_{\ep'},\quad \forall I\in\scrA.
\end{equation*}
\end{prop}

\begin{lem}\label{lemSymConvUptoSca}
For any $\ep'>0$ and any sequence $\{k_nu_n\}$ with $u_n \in B_{R_n,\ep'}$ and $R_n \to + \infty$, we have $\lim_n (g_n)_*[\mu_H]=[\mu_{G/\Gamma}]$.
\end{lem}

\begin{proof}
 If not true, by Proposition \ref{propToriSym}, there exists $x_{\neq 0}$  in the Lie algebra of $H$ such that $\Ad(g_n)\cdot x$ is bounded. 
 Write $x=\diag(x_1,...,x_N,x_{N+1},...,x_{2N})$ with $x_{N+i}=-x_{N+1-i}$.
 The absolute value of $(i,i+1)$-th entry of $\Ad(g_n)\cdot x$ is equal to $|x_{i+1}-x_i||(u_n)_{i,i+1}| \geq |x_{i+1}-x_i|R\ep'$, which diverges as $R \to \infty$. This is a contradiction.
\end{proof}

Recall that $\Omega_{g,\eta}$(see section \ref{secNondivMaxSplTori}) is equal to
\begin{equation*}
    \{
    t \in \Lie(H)\,\vert\,
    \diffalpha_I(t) \geq \ln{\eta} - \ln{||ge_I||},\,
    \forall I\in\scrA
    \}.
\end{equation*}
We take $\eta>0$ small enough depending on $f$ such that 
\begin{equation*}
     \int_H f(kuh) \mu_H(h) = \int_{\exp{(\Omega_{u,\eta})}} f(kuh) \mu_H(h)
\end{equation*}
for all $u$ and $k$.

\begin{proof}[Proof of Lemma \ref{lemSymCountMainTerm}]
In light of Lemma \ref{lemSymConvUptoSca}, it remains to show that for $\eta>0$ sufficiently small and $u_R\in B_{R,\ep'}$,
\begin{equation*}
    \lim_R \frac{\mu_H(\Omega_{u_R,\eta})}{C_1\ln{R}^{N}} =1.
\end{equation*}
By Proposition \ref{PropZhengEstimate} above, 
\begin{equation*}
    \Omega_{u_R,\eta} \subset 
    \{
    t\in \Lie(H) \,\vert\,
    \diffalpha_I(t) \geq \ln{\eta} - c_I \ln{R}- M_{\ep'},\,\forall I\in \scrA
    \}
\end{equation*}
and 
\begin{equation*}
    \Omega_{u_R,\eta} \supset 
    \{
    t\in \Lie(H) \,\vert\,
    \diffalpha_I(t) \geq \ln{\eta} - c_I \ln{R} + M_{\ep'},\,\forall I\in \scrA
    \}.
\end{equation*}
By dividing $\ln R$, we get
\begin{equation*}
\begin{aligned}
    & \{t\in \Lie(H),\,  \diffalpha_I(t) \geq  -c_I + \frac{\ln{\eta}-M_{\ep'}}{\ln R}
    ,\,\forall I\in \scrA
    \} \\
      &\subset \frac{\Omega_{u_R,\eta}}{\ln{R}} \subset
      \{t\in \Lie(H),\,  \diffalpha_I(t) \geq  -c_I + \frac{\ln{\eta}+M_{\ep'}}{\ln R}
    ,\,\forall I\in \scrA
    \} .
\end{aligned}
\end{equation*}
By the definition of $C_1$, we are done.
\end{proof}

\begin{proof}[Proof of Lemma \ref{lemSymCountErrorTerm}]
By Proposition \ref{PropZhengEstimate} above, for $u\in B_R$,
\begin{equation*}
     \frac{\Omega_{u,\eta}}{\ln{R}} \subset 
     \{t\in \Lie(H),\, 
      \diffalpha_I(t) \geq  \frac{\ln \eta }{\ln{R}} - M
     ,\,\forall I\in \scrA
     \}.
\end{equation*}
So $\mu_H(\Omega_{u,\eta}) \leq M' \ln{R}^N$ for some constant $M'>0$. By taking $C'_f:=\sup |f|$, then 
\begin{equation*}
     \left| \frac{1}{C_1(\ln{R})^N} \int f(kuh) \mu_H(h)  \right| \leq 
     \frac{C'_f\mu_H(\Omega_{u,\eta})}{C_1(\ln{R})^N} \leq C_f
\end{equation*}
for some constant $C_f>0$.
\end{proof}

\subsection{Shearing a divergent geodesic}

In this subsection we prove Theorem \ref{thmShearGeod}.
Let $\bmG=\SO_{Q}$ with $Q(x_1,...,x_n,y)=x_1^2+...+x_n^2-y^2$. 
We start by reviewing some hyperbolic geometry for which the reader is referred to  \cite[Chapter A]{BenePetr92} for details.

\subsubsection*{Lift from $\bbH^n$ to the group}
The group $G$ naturally acts transitively on $\bbI^n:=\{Q=-1\}$, a subset of $\R^{n+1}$. There is a bijection $\Phi$ from $\bbI^n$ to $\bbH^n$ given by the composition of 
\begin{equation*}
     \begin{aligned}
           &(x_1,...,x_n,y) \mapsto \frac{(x_1,...,x_n)}{1+y} \\
           &(x_1,...,x_n) \mapsto \frac{2(x+\bme_n)}{||x+\bme_n||^2} -\bme_n.
     \end{aligned}
\end{equation*}
Therefore the action of $G$ on $\bbI^n$ can be transported to an action on $\bbH^n$, which turns out to be isometric. Let $\bmo:=\Phi((0,...,0,1))$ and the stabilizer of $\bmo$ in $G$ is a maximal compact subgroup $K$ in $G$. The map $g \mapsto g\cdot \bmo$ induces an isometry between $G/K$ and $\bbH^n$. So it naturally descends to an isometry $\Psi$ between $\Gamma\bs G/K $ and $\Gamma\bs\bbH^n$.

For $t\in\R$, let
\begin{equation*}
    a_t:= \left[
\begin{array}{c|c}
I_{n-1} &  0
 \\ \hline
0  & 
\begin{array}{ccc}
      \cosh{t} & \sinh{t}\\
      \sinh{t} & \cosh{t}  
\end{array}
\end{array}
\right]
\end{equation*}
and for $\bmv=(v_1,...,v_{n-1})$,
\begin{equation*}
u_{\bmv}=\left[
\begin{array}{c|c}
I_{n-1} &  
\begin{array}{cc}
    -v_1 & v_1 \\
    \vdots & \vdots\\
    -v_{n-1} & v_{n-1}
\end{array}
 \\ \hline
\begin{array}{ccc}
    v_1 & \cdots & v_{n-1} \\
    v_{1} & \cdots & v_{n-1}
\end{array}  
& 
\begin{array}{ccc}
      1- \frac{||v||^2 }{2}&\frac{||v||^2 }{2}\\
      - \frac{||v||^2 }{2} & 1 +\frac{||v||^2 }{2} 
\end{array}
\end{array}\right].    
\end{equation*}
Let $\bmU$ be the $\Q$-subgroup generated by $u_{\bmv}$ as $\bmv$ varies in $\Q^{n-1}$. This is the commutative unipotent radical of a proper $\Q$-parabolic subgroup of $\bmG$ that is contracted by $a_t \bullet a_t^{-1}$ as $t \to -\infty$. 
Note that $\bmv \mapsto u_{\bmv}$ gives an isomorphism between $(\Q^n,+)$ and $\bmU$ as $\Q$-algebraic groups.

Under $\Psi$, $\Gamma a_t u_{\bmv} K$ is sent to the projection of $\Phi(a_tu_{\bmv}\cdot \bme_{n+1})=a_tu_{\bmv}\cdot \bmo$, which is equal to
\begin{equation*}
    \frac{2+2\cosh{t}+||\bmv||^2e^t}{||\bmv||^2+(e^t+1 + ||\bmv||^2e^t)^2} \cdot (v_1,...,v_{n-1},1).
\end{equation*}
Now let $t$ varies over $\R$, then we find that $\{\Psi[a_t u_{\bmv}],\,t\in\R\}=\pi_{\Gamma}(\calI_{\bmv})$. In light of this calculation it suffices to prove the following:
\begin{thm}
Assume that $\bmv$ is not contained in any proper $\Q$-linear subspace. 
Then under the Chabauty topology, $\{\pi_{\Gamma}(a_t u_{s\bmv}),\, t\in\R\}$ converges to $\Gamma\bs G$ as $s$ tends to infinity.
\end{thm}

\begin{proof}
Without loss of generality assume that $||\bmv||=1$.
Recall Corollary \ref{thmConChabau} and use the terminology there. Take $g_k=u_{s_k\bmv}$ for arbitrary $s_k \to \infty$ and $\bmH$ to be the group generated by $a_t$, it suffices to show $\bmL=\bmG$ possibly after passing to a subsequence. To find $\bmL$ we will modify $\{g_k\}$ from left by a bounded sequence to some sequence contained in $\Gamma$. 
Use $\bmv \mapsto u_{\bmv}$ to identify $2\Z^{n-1}$ with a finite index subgroup of $ \bmU(\Z)$.
As $\Gamma$ is commensurable with $\bmG(\Z)$, $\Gamma \cap 2\Z^{n-1}$ is a finite index subgroup of $2\Z^{n-1}$. And therefore one can find a positive integer $N_0$ such that $\Gamma$ contains $2N_0\Z^{n-1}$. 

Take a sequence $\{\bmv^k\} $ in $2N_0\Z^{n-1}$ that is of bounded distance from $s_k\bmv$. Then $u_{s_k\bmv}$ is also of bounded distance (both from left and from right) from $u_{\bmv^k}$, which are contained in $\Gamma$. By our choice, we have that 
\begin{equation*}
    \lim_k \frac{\bmv^k}{||\bmv^k||} = \bmv.
\end{equation*}
Passing to a subsequence we may assume that $\bmL$ contains $u_{\bmv^k}^{-1}\bmH u_{\bmv^k}$ for all $k$. Let us compute $u_{\bmv^k}^{-1}a_t u_{\bmv^k}=u_{-\bmv^k}a_t u_{\bmv^k}$, which is equal to
\begin{equation*}
          \left[
\begin{array}{c|c}
I_{n-1} &  
\begin{array}{cc}
    -(1-e^{-t})v^k_1 & (1-e^{-t})v^k_1 \\
    \vdots & \vdots\\
    -(1-e^{-t})v^k_{n-1} & (1-e^{-t})v^k_{n-1}
\end{array}
 \\ \hline
\begin{array}{ccc}
    (e^t-1)v^k_1 & \cdots & (e^t-1)v^k_{n-1} \\
    (e^t-1)v^k_{1} & \cdots & (e^t-1)v^k_{n-1}
\end{array}  
& 
A_k
\end{array}\right]    
\end{equation*}
with 
\begin{equation*}
    A_k =\left[ \begin{array}{ccc}
      \cosh{t}+(1-\cosh{t})||\bmv^k||^2 &  \sinh{t}-(1-\cosh{t})||\bmv^k||^2\\
      \sinh{t}+(1-\cosh{t})||\bmv^k||^2 & \cosh{t}-(1-\cosh{t})||\bmv^k||^2
\end{array}\right] .
\end{equation*}
As $||\bmv^k|| \to \infty$, for any $\lambda \in \R $ we can find $t_k\to 0$ such that 
\begin{equation*}
    (e^{t_k}-1)||\bmv^k|| = \lambda.
\end{equation*}
Then  $ \lim_k u_{\bmv^k}^{-1}a_{t_k} u_{\bmv^k}$ is equal to
\begin{equation*}
    \left[
\begin{array}{c|c}
I_{n-1} &  
\begin{array}{cc}
    -\lambda v_1 & \lambda v_1 \\
    \vdots & \vdots\\
    -\lambda v_{n-1} & \lambda v_{n-1}
\end{array}
 \\ \hline
\begin{array}{ccc}
    \lambda v_1 & \cdots & \lambda v_{n-1} \\
    \lambda v_{1} & \cdots & \lambda v_{n-1}
\end{array}  
& 
\begin{array}{ccc}
     1- \frac{\lambda^2}{2} &   \frac{\lambda^2}{2} \\
     -\frac{\lambda^2}{2} &  1 + \frac{\lambda^2}{2}
\end{array}
\end{array}\right] 
\end{equation*}
which is exactly $u_{\lambda \bmv}$. So $\bmL$ contains $u_{\lambda \bmv}$ for all $\lambda\in\R$. As $\bmL$ is defined over $\Q$, it contains the smallest $\Q$-subgroup containing $u_{\lambda \bmv}$ which,  by our assumption on $\bmv$, is equal to $\bmU$.  As $\bmL$ is reductive, the Adjoint action of $a_t$ on the Lie algebra of $\bmL$ is semisimple. As the contracting horospherical subgroup $\bmU$ (with respect to $\Ad(a_t)$ with $t\to -\infty$) is already contained in $\bmL$, the expanding horospherical subgroup should also be contained in $\bmL$, for otherwise the action would not be semisimple. But these two together generate $\bmG$. So $\bmG=\bmL$.
\end{proof}

\section*{Acknowledgement}
We are grateful to discussions with Nimish Shah, Pengyu Yang, Han Zhang and Cheng Zheng.

\bibliographystyle{amsalpha}
\bibliography{ref}

\end{document}